\documentclass[10.5pt,onecolumn]{IEEEtran}

\usepackage{algorithm}
\usepackage{algorithmic}
\usepackage{geometry}

\usepackage{bm,amsmath,amssymb,amsfonts,graphicx,epsfig,amsthm,color}
\usepackage{bbold}

\usepackage[hyphens]{url}
\usepackage{hyperref}
\hypersetup{colorlinks=false}

\usepackage[depth=-1]{bookmark}

\usepackage{booktabs}       
\usepackage{amsfonts}       
\usepackage{nicefrac}       
\usepackage{microtype}      

\usepackage{times}
\usepackage{graphicx} 
\usepackage{subfigure} 

\usepackage{wrapfig}

\usepackage{accents}
\makeatletter
\def\wid{\check{{\cc@style\underline{\mskip9.5mu}}}}
\def\Wideubar{\underaccent{{\cc@style\underline{\mskip6mu}}}}
\makeatother

\makeatletter
\def\wideubar{\underaccent{{\cc@style\underline{\mskip9.5mu}}}}
\def\Wideubar{\underaccent{{\cc@style\underline{\mskip6mu}}}}
\makeatother

\makeatletter
\def\widebar{\accentset{{\cc@style\underline{\mskip9.5mu}}}}
\def\Widebar{\accentset{{\cc@style\underline{\mskip6mu}}}}
\makeatother

\newtheorem{proposition}{Proposition}
\newtheorem{lemma}{Lemma}

\newtheorem{theorem}{Theorem}

\theoremstyle{remark}

\newcommand{\eqdef}{\overset{d}{=}}

\newcommand{\minimize}{{\rm minimize}}
\newcommand{\maximize}{{\rm maximize}}

\allowdisplaybreaks

\begin{document}
\title{Solving Almost all Systems of Random Quadratic Equations}

\author{
Gang Wang,
	Georgios B. Giannakis, Yousef Saad, and Jie Chen

\thanks{Work in this paper was supported partially by NSF grants 1500713 and 1514056.  
G. Wang and G. B. Giannakis are with the Digital Technology Center and the Department of Electrical and Computer Engineering, University of Minnesota, Minneapolis, MN 55455, USA. G. Wang is also with the State Key Lab of Intelligent Control and Decision of Complex Systems, Beijing Institute of Technology, Beijing 100081, P. R. China. Y. Saad is with the Department of Computer Science and Engineering, University of Minnesota, Minneapolis, MN 55455, USA.
J. Chen is with the State Key Lab of Intelligent Control and Decision of Complex Systems, Beijing Institute of Technology, Beijing 100081, P. R. China. Emails: \{gangwang,\,georgios,\,saad\}@umn.edu; chenjie@bit.edu.cn.}


}


\maketitle

\begin{abstract}
This paper deals with finding an $n$-dimensional solution $\bm{x}$ to a system of quadratic equations of the form $y_i=|\langle\bm{a}_i,\bm{x}\rangle|^2$ for $1\le i \le m$, which is also known as phase retrieval and is NP-hard in general. We put forth a novel procedure for minimizing the amplitude-based least-squares empirical loss, that starts with a weighted maximal correlation initialization obtainable with a few power or Lanczos iterations, followed by successive refinements based upon a sequence of iteratively reweighted (generalized) gradient iterations. The two (both the initialization and gradient flow) stages distinguish themselves from prior contributions by the inclusion of a fresh (re)weighting regularization technique. The overall algorithm is conceptually simple, numerically scalable, and easy-to-implement. For certain random measurement models, the novel procedure is shown capable of finding the true solution $\bm{x}$ in time proportional to reading the data $\{(\bm{a}_i;y_i)\}_{1\le i \le m}$. This holds with high probability and without extra assumption on the signal $\bm{x}$ to be recovered, provided that the number $m$ of equations is some constant $c>0$ times the number $n$ of unknowns in the signal vector, namely, $m>cn$. Empirically, the upshots of this contribution are: i) (almost) $100\%$ perfect signal recovery in the high-dimensional (say e.g., $n\ge 2,000$) regime given only an \emph{information-theoretic limit} number of noiseless equations, namely, $m=2n-1$ in the real-valued Gaussian case; and, ii) (nearly) optimal statistical accuracy in the presence of additive noise of bounded support. 
Finally, substantial numerical tests using both synthetic data and real images corroborate markedly improved signal recovery performance and computational efficiency of our novel procedure relative to state-of-the-art approaches. 

\end{abstract}

%


\section{Introduction}
\label{sec:intro}

One is often faced with solving quadratic equations of the form $y_i=|\langle\bm{a}_i,\bm{x}\rangle|^2$, or equivalently, 
\begin{equation}
\label{eq:quad}
\psi_i=|\langle \bm{a}_i,\bm{x}\rangle|,\quad 1\le i\le m
\end{equation}
where $\bm{x}\in\mathbb{R}^n$ 
is the wanted unknown $n\times 1$ signal, while given observations $\psi_i$ and feature/sensing vectors $\bm{a}_i\in\mathbb{R}^n$ that are collectively stacked in the data vector $\bm{\psi}:=[\psi_i]_{1\le i\le m}$ and the $m\times n$ sensing matrix $\bm{A}:=[\bm{a}_i]_{1\le i\le m}$, respectively. Put differently, given information about the (squared) modulus of the inner products of the signal vector $\bm{x}$ and several known measurement vectors $\bm{a}_i$, can one reconstruct exactly (up to a global sign) $\bm{x}$, or alternatively, the missing signs of $\langle\bm{a}_i,\bm{x}\rangle$? In fact, much effort has recently been devoted to determining the number of such equations necessary and/or sufficient for the uniqueness of the solution $\bm{x}$; see, for instance, \cite{2006balan}, \cite{savephase}. It has been proved that a number $m\ge 2n-1$ of generic \footnote{It is out of the scope of the present paper to explain the meaning of generic vectors, whereas interested readers are referred to~\cite{2006balan}.} (which includes the case of random vectors) measurement vectors $\bm{a}_i$ are sufficient for uniquely determining an $n$-dimensional real vector $\bm{x}$, while $m=2n-1$ is shown also necessary \cite{2006balan}. In this sense, the number $m=2n-1$ of equations as in \eqref{eq:quad}
 can be thought of as the information-theoretic limit for such a quadratic system to be solvable. Nevertheless, even for random measurement vectors, despite the existence of a unique solution given the minimal number $2n-1$ of quadratic equations, it is unclear so far whether there is a numerical polynomial-time algorithm that is able to stably find the true solution (say e.g., with probability $\ge 99\%$)?  

In diverse physical sciences and engineering fields, it is impossible or very difficult to record phase measurements. The problem of recovering the signal or phase from magnitude measurements only, also commonly known as phase retrieval, emerges naturally 
\cite{1982fienup,gerchberg,spm2015eldar}. Relevant application domains include e.g., X-ray crystallography, ptychography, astronomy, and coherent diffraction imaging \cite{spm2015eldar}. In such setups, optical measurement and detection systems record solely the photon flux, which is proportional to the (squared) magnitude of the field, but not the phase.
Problem \eqref{eq:quad} in its squared form, on the other hand, can be readily recast as an instance of nonconvex quadratically constrained quadratic programming (QCQP), that subsumes as special cases several well-known combinatorial optimization problems involving Boolean variables, e.g., the NP-complete stone problem \cite[Sec. 3.4.1]{book2001nemirovski}.
A related task of this kind is that of estimating the mixture of linear regressions, where the latent membership indicators can be converted into the missing phases  \cite{mixedlr}.
Although of simple form and practical relevance across different fields, solving systems of nonlinear equations is arguably the most difficult problem in all of the numerical computations \cite[Page 355]{1992book}.

Regarding common notation used throughout this paper, lower- (upper-) case boldface letters denote vectors (matrices). Calligraphic letters are reserved for sets, e.g., $\mathcal{S}$. Symbol $A/B$ means either $A$ or $B$, while fractions are denoted by $\nicefrac{A}{B}$. 
The floor operation $\lfloor c\rfloor$ gives the largest integer no greater than the given number $c>0$, $|\mathcal{S}|$ the number of entries in set $\mathcal{S}$, and $\|\bm{x}\|$ is the Euclidean norm. Since $\bm{x}\in\mathbb{R}^n$ and $-\bm{x}$ are indistinguishable given $\{\psi_i\}$, let ${\rm dist}(\bm{z},\bm{x})=\min\{\|\bm{z}+\bm{x}\|,\|\bm{z}-\bm{x}\|\}$ be the Euclidean distance of any estimate $\bm{z}\in\mathbb{R}^n$ to the solution set $\{\pm\bm{x}\}$ of \eqref{eq:quad}.     

\subsection{Prior contributions}
Following the least-squares criterion (which coincides with the maximum likelihood one assuming additive white Gaussian noise), the problem of solving systems of quadratic equations can be naturally recast as the ensuing empirical loss minimization
\begin{equation}\label{eq:ls}
\underset{\bm{z}\in\mathbb{R}^n}{\minimize}~\,\,
L(\bm{z}):=
\frac{1}{m}\sum_{i=1}^m \ell(\bm{z};\psi_i/y_i)
\end{equation}
where one can choose to work with the \emph{amplitude-based} loss function $\ell(\bm{z};\psi_i):=
\nicefrac{(\psi_i-|\langle \bm{a}_i,\bm{z}\rangle|)^2}{2}$ \cite{taf}, or the \emph{intensity-based} ones $\ell(\bm{z};y_i):=
\nicefrac{(y_i-|\langle \bm{a}_i,\bm{z}\rangle|^2)^2}{2}$ \cite{wf}, \cite{acha2014yesm}, and its related Poisson likelihood $\ell(\bm{z};y_i):=y_i\log(|\langle\bm{a}_i,\bm{z}\rangle|^2)-|\langle\bm{a}_i,\bm{z}\rangle|^2$ \cite{twf}. 
Either way, the objective functional $L(\bm{z})$ is rendered nonconvex; hence, it is in general NP-hard and computationally intractable to compute the least-squares or the maximum likelihood estimate (MLE) \cite{book2001nemirovski}. 

Minimizing the squared modulus-based least-squares loss in \eqref{eq:ls}, several numerical polynomial-time algorithms have been devised based on convex programming for certain choices of design vectors $\bm{a}_i$
\cite{
	 array,
	phaselift, 
phasecut,tit2015ccg,fcm2014candes,acha2015krt}. 
Relying upon the so-called matrix-lifting technique (a.k.a, Shor's relaxation), semidefinite programming (SDP) based convex approaches first express all intensity data into linear terms in a new rank-$1$ matrix variable, followed by solving a convex SDP after dropping the rank constraint (a.k.a., semidefinite relaxation). It has been established that perfect recovery and (near-)optimal statistical accuracy can be achieved in noiseless and noisy settings respectively 
with an optimal-order number of measurements \cite{fcm2014candes}. In terms of computational efficiency however, convex approaches entail storing and solving for an $n\times n$ semi-definite matrix, whose worst-case computational complexity scales as $n^{4.5}\log \nicefrac{1}{\epsilon}$ provided that the number $m$ of constraints is on the order of the dimension $n$ \cite{phasecut}, which does not scale well to high-dimensional tasks. Another recent line of convex relaxation \cite{phasemax}, \cite{phasemax1}, 
\cite{phasemaxproof} 
reformulated the problem of phase retrieval as that of sparse signal recovery, and solves a linear program in the natural parameter vector domain. 
Although exact signal recovery can be established assuming an accurate enough anchor vector, its empirical performance is not competitive with state-of-the-art phase retrieval approaches.    

Instead of convex relaxation, recent proposals also
advocate
judiciously initialized iterative procedures for coping with certain nonconvex formulations directly, which include solvers based on e.g., alternating minimization \cite{altmin}, Wirtinger flow \cite{wf}, \cite{twf}, \cite{reshaped}, 
 \cite{mtwf}, \cite{arxiv2016be}, \cite{pwf}, \cite{gu2017}, amplitude flow \cite{taf,staf,sparta,nips2016wg}, as well as a prox-linear procedure via composite optimization \cite{duchi2017}, \cite{duchi2017stochastic}. 
These nonconvex approaches operate directly upon vector optimization variables, therefore leading to significant computational advantages over matrix-lifting based convex counterparts. With random features, they can be interpreted as performing stochastic optimization over acquired data samples $\{(\bm{a}_i;\psi_i/y_i)\}_{1\le i \le m}$ to approximately minimize the population risk functional $\widebar{L}(\bm{z}):=\mathbb{E}_{(\bm{a}_i,\psi_i/y_i)}[\ell(\bm{z};\psi_i/y_i)]$. It is well documented that minimizing nonconvex functionals is computationally
intractable in general due to existence of multiple stationary points \cite{book2001nemirovski}. Assuming random Gaussian sensing vectors however, such nonconvex paradigms can provably locate the global optimum under suitable conditions, some of which also achieve optimal (statistical) guarantees.  
Specifically, starting with a judiciously designed initial guess,
successive improvement is effected based upon a sequence of (truncated) (generalized) gradient iterations given by
\begin{equation}
\label{eq:update}
\bm{z}^{t+1}:=\bm{z}^t-\frac{\mu^t}{m}\sum_{i\in\mathcal{T}^{t+1}}\!\!\!\nabla \ell_i(\bm{z}^t;\psi_i/y_i),\quad t=0,\,1,\,\ldots
\end{equation}
where $\bm{z}^t$ denotes the estimate returned by the algorithm 
at the $t$-th iteration, $\mu^t>0$ the learning rate, 
and $\nabla\ell(\bm{z}^t,\psi_i/y_i)$ is the (generalized) gradient of the modulus- or squared modulus-based least-squares loss evaluated at $\bm{z}^t$ \cite{clarke1975gg}. Here, $\mathcal{T}^{t+1}$ represents some time-varying index set signifying and effecting the per-iteration truncation.

Although they achieve optimal statistical guarantees in both noiseless and noisy settings, state-of-the-art (convex and nonconvex) approaches studied under random Gaussian designs, empirically require stable recovery of a number of equations (several) times larger
than the aforementioned information-theoretic limit \cite{twf,wf,reshaped}. As a matter of fact, when there are numerous enough measurements (on the order of the signal dimension $n$ up to some polylog factors), 
the modulus-square based least-squares loss functional admits benign geometric structure in the sense that \cite{sun2016}: with high probability, i) all local minimizers are global; and, ii) there always exists a negative directional curvature at every saddle point. In a nutshell,
the grand challenge of solving systems of random quadratic equations
remains to develop numerical polynomial-time algorithms capable of achieving perfect recovery and optimal statistical accuracy when the number of measurements approaches the information-theoretic limit.

\subsection{This work}

Building upon but going beyond the scope of the aforementioned nonconvex paradigms, 
the present paper puts forward a novel iterative linear-time procedure, namely, time proportional to that required by the processor to scan the entire data $\{(\bm{a}_i;\psi_i)\}_{1\le i\le m}$, which we term \emph{reweighted amplitude flow} and abbreviate as RAF. Our methodology is capable of solving noiseless random quadratic equations exactly, and of constructing an estimate of (near)-optimal statistical accuracy from noisy modulus observations. Exactness and accuracy hold with high probability and without extra assumption on the signal $\bm{x}$ to be recovered, provided that the ratio $m/n$ of the number of measurements to that of the unknowns exceeds some large constant. Empirically, our approach is demonstrated to be able to achieve perfect recovery of arbitrary high-dimensional signals given a \emph{minimal} number of equations, which in the real case is $m=2n-1$. The new twist here is to leverage judiciously designed yet conceptually simple (iterative)(re)weighting regularization techniques to enhance existing initializations and also gradient refinements. 
An informal depiction of our RAF methodology is given in two stages below, with rigorous algorithmic details deferred to Section \ref{sec:main}: 

\begin{enumerate}
	\item[\textbf{S1)}] \textbf{Weighted maximal correlation initialization:} Obtain an initializer $\bm{z}^0$ maximally correlated with a carefully selected subset $\mathcal{S}\subsetneqq \mathcal{M}:=\{1,2,\ldots,m\}$ of feature vectors $\bm{a}_i$, whose contributions toward constructing $\bm{z}^0$ are judiciously weighted by suitable parameters $\{w_{i}^0>0\}_{i\in\mathcal{S}}$.
	\item[\textbf{S2)}] \textbf{Iteratively reweighted ``gradient-like'' iterations:} Loop over 
	$0\le t\le T$
	\begin{equation}\vspace{-5pt}
	\bm{z}^{t+1}=\bm{z}^t-\frac{\mu^t}{m}\sum_{i=1}^m w_{i}^t\,\nabla\ell(\bm{z}^t;\psi_i)
	\end{equation}
	for some time-varying weights $w_{i}^t\ge 0$ that are adaptive in time, each depending on the current iterate $\bm{z}_t$ and the datum $(\bm{a}_i;\psi_i)$.  
\end{enumerate}

Two attributes of our novel methodology are worth highlighting. First, albeit being a variant of the orthogonality-promoting initialization \cite{taf}, the initialization here [cf. S1)] is distinct in the sense that different importance is attached to each selected datum $(\bm{a}_i;\psi_i)$, or more precisely, to each selected directional vector $\bm{a}_i$. Likewise, the gradient flow [cf. S2)] weights judiciously the search direction suggested by each datum $(\bm{a}_i;\psi_i)$. In this manner, more accurate and robust initializations as well as more stable overall search directions in the gradient flow stage can be obtained even based only on a relatively limited number of data samples. Moreover, with particular choices of weights $w_{i}^t$'s (for example, when they take $0/1$ values), our methodology subsumes as special cases the recently proposed truncated amplitude flow (TAF) \cite{taf} and reshaped Wirtinger flow (RWF) \cite{reshaped}.

 The reminder of this paper is structured as follows. The two stages of our RAF algorithm are motivated and developed in Section~\ref{sec:alg}, and
 Section~\ref{sec:main} summarizes the algorithm and establishes its theoretical performance. 
 Extensive numerical tests evaluating the performance of RAF relative to state-of-the-art approaches are presented in Section~\ref{sec:test}, while useful technical lemmas and main proof ideas are given in Section~\ref{sec:proof}.  Concluding remarks are drawn in \ref{sec:con}, and technical proof details are provided in the Appendix at the end of the paper.

\section{Algorithm: Reweighted Amplitude Flow}
\label{sec:alg}

This section explains the intuition and the basic principles behind each stage of RAF in detail. For theoretical concreteness, we focus on the \emph{real-valued Gaussian model} with a real signal $\bm{x}$, and independent Gaussian random measurement vectors $\bm{a}_i\sim \mathcal{N}(\bm{0},\bm{I})$, $1\le i\le m$. Nevertheless, our RAF approach can be applied without algorithmic changes even when the complex-valued Gaussian with $\bm{x}\in\mathbb{C}^n$ and independent $\bm{a}_i\sim\mathcal{CN}(\bm{0},\bm{I}_n):=\mathcal{N}(\bm{0},\bm{I}_n/2)+j\mathcal{N}(\bm{0},\bm{I}_n/2)$, and also the coded diffraction pattern (CDP) models are considered. 

\subsection{Weighted maximal correlation initialization}
\label{subsec:initial}

A key enabler of general nonconvex iterative heuristics' success in finding the global optimum is to seed them with an excellent starting point \cite{tit2010spectral}.     
In fact, several smart initialization strategies have been advocated for iterative phase retrieval algorithms; see e.g., the spectral \cite{altmin}, 
\cite{wf}, truncated spectral \cite{twf}, \cite{reshaped}, and orthogonality-promoting \cite{taf} initializations. 
One promising approach among them is the one proposed in \cite{taf}, which is robust to outliers \cite{duchi2017}, and also enjoys better phase transitions than the spectral procedures \cite{lu2017transition}.  
To hopefully achieve perfect signal recovery at the information-theoretic limit however, its numerical performance may still need further enhancement. On the other hand, it is intuitive that to improve the initialization performance (over state-of-the-art schemes) becomes increasingly challenging as the number of acquired data samples approaches the information-theoretic limit of $m=2n-1$. 

In this context, we develop a more flexible initialization scheme based on the correlation property (as opposed to the orthogonality), in which the added benefit relative to the initialization procedure in \cite{taf} is the inclusion of a flexible weighting regularization technique to better balance the useful information exploited in all selected data. In words, we introduce judiciously designed weights to the initialization procedure developed in \cite{taf}. 
Similar to related approaches, our strategy entails estimating both the norm $\|\bm{x}\|$ and the unit direction $\nicefrac{\bm{x}}{\|\bm{x}\|}$. Leveraging the strong law of large numbers and the rotational invariance of Gaussian $\bm{a}_i$ sampling vectors (the latter suffices to assume $\bm{x}=\|\bm{x}\|\bm{e}_1$, with $\bm{e}_1$ being the first canonical vector in $\mathbb{R}^n$), it is clear that
\begin{equation}\label{eq:normest}
\frac{1}{m}	\sum_{i=1}^m\psi_i^2=\frac{1}{m}\sum_{i=1}^m \big|\langle \bm{a}_i, \|\bm{x}\|\bm{e}_1\rangle \big|^2=\Big(\frac{1}{m}\sum_{i=1}^m a_{i,1}^2\Big)\|\bm{x}\|^2\approx \|\bm{x}\|^2
\end{equation} 
whereby $\|\bm{x}\|$ can be estimated to be $\sum_{i=1}^m\!\nicefrac{\psi_i^2}{m}$. This estimate proves very accurate even with an information-theoretic limit number of data samples because it is unbiased and tightly concentrated.

The challenge thus lies in accurately estimating the direction of $\bm{x}$, or seeking a unit vector maximally aligned with $\bm{x}$, which is a bit tricky. 
To gain intuition into our initialization strategy, let us first present a variant of the initialization in \cite{taf}, whose robust counterpart to outlying measurements has been recently discussed in \cite{duchi2017}.
Note that the larger the modulus $\psi_i$ of the inner-product between $\bm{a}_i$ and $\bm{x}$ is, the known design vector $\bm{a}_i$ is deemed \emph{more correlated} to the unknown solution $\bm{x}$, hence bearing useful directional information of $\bm{x}$. Inspired by this fact and based on available data $\{(\bm{a}_i;\psi_i)\}_{1\le i\le m}$, one can sort all (absolute) correlation coefficients $\{\psi_i\}_{1\le i\le m}$ in an ascending order, to yield ordered coefficients denoted by $0<\psi_{[m]} \le \cdots\le \psi_{[2]}\le\psi_{[1]}$. Sorting $m$ records takes time proportional to $\mathcal{O}(m\log m)$. \footnote{$f(m)=\mathcal{O}(g(m))$ means that there exists a constant $C>0$ such that $|f(m)|\le C|g(m)|$.} Let $\mathcal{S}\subsetneqq\mathcal{M}$ represent the set of selected feature vectors $\bm{a}_i$ to be used for computing the initialization, which is to be designed next. Fix \emph{a priori} the cardinality $|\mathcal{S}|$ to some integer on the order of $m$, say, $|\mathcal{S}|:=\lfloor \nicefrac{3m}{13}\rfloor$.
It is then natural to \emph{define} $\mathcal{S}$ to 
collect the $\bm{a}_i$ vectors that correspond to one of the largest $|\mathcal{S}|$ correlation coefficients $\{\psi_{[i]}\}_{1\le i\le |\mathcal{S}|}$, each of which can be thought of as pointing to (roughly) the direction of $\bm{x}$. Approximating the direction of $\bm{x}$ thus boils down to finding a vector to maximize its correlation with the subset $\mathcal{S}$ of selected directional vectors $\bm{a}_i$. Succinctly, the wanted approximation vector can be efficiently found as the solution of 
\begin{equation}\label{eq:mci}
\underset{\|\bm{z}\|=1}{\maximize}~\,\frac{1}{|\mathcal{S}|}\sum_{i\in\mathcal{S}} \big|\langle\bm{a}_i,\bm{z}\rangle\big|^2=
\bm{z}^\ast\Big(\frac{1}{|\mathcal{S}|}\sum_{i\in\mathcal{S}}\bm{a}_i\bm{a}_i^\ast\Big)\bm{z}
\end{equation} 
where the superscript $^\ast$ represents the transpose. Upon scaling the solution of \eqref{eq:mci} by the norm estimate $\sum_{i=1}^m\!\nicefrac{\psi_i^2}{m}$ in \eqref{eq:normest}
to match the size of $\bm{x}$, 
we obtain what we will henceforth refer to as maximal correlation initialization.

As long as $|\mathcal{S}|$ is chosen on the order of $m$, the maximal correlation method outperforms the spectral ones in \cite{wf,altmin,twf}, and has comparable performance to the orthogonality-promoting method \cite{taf}.
Its empirical performance 
around the information-theoretic limit however, is still not the best that we can hope for. Observe that all directional vectors $\{\bm{a}_i\}_{i\in\mathcal{S}}$ selected for forming the matrix $\overline{\bm{Y}}:=\frac{1}{|\mathcal{S}|}\sum_{i\in\mathcal{S}}\bm{a}_i\bm{a}_i^\ast$ in \eqref{eq:mci} are treated \emph{the same} in terms of their contributions to constructing the (direction of the) initialization. Nevertheless, according to our starting principle, this ordering information carried by the selected $\bm{a}_i$ vectors has \emph{not} been exploited by the initialization scheme in \eqref{eq:mci} (see also \cite{taf}, \cite{duchi2017}). In words, 
if for selected data $i,j\in\mathcal{S}$, the correlation coefficient of $\psi_i$ with $\bm{a}_i$ is larger than that of $\psi_j$ with $\bm{a}_j$, then $\bm{a}_i$ is deemed more correlated (with $\bm{x}$) than $\bm{a}_j$ is, hence bearing more useful information about the wanted direction of $\bm{x}$. It is thus prudent to weight more (i.e., attach more importance to) the selected $\bm{a}_i$ vectors corresponding to larger $\psi_i$ values. Given the ordering information $\psi_{[|\mathcal{S}|]}\le \cdots\le \psi_{[2]}\le \psi_{[1]}$ available from the sorting procedure, a natural way to achieve this goal is weighting each $\bm{a}_i$ vector with simple functions of $\psi_i$, say e.g., taking the weights $w_{i}^0:=\psi_i^\gamma,~\forall i\in \mathcal{S}$, with the exponent parameter $\gamma\ge 0$ chosen to maintain the wanted ordering $w_{|\mathcal{S}|}^{0}\le \cdots\le w_{[2]}^0\le w_{[1]}^0$. 
In a nutshell, a more flexible initialization scheme, that we refer to as \emph{weighted maximal correlation},
can be summarized as follows
\begin{equation}
\label{eq:rmci}
\tilde{\bm{z}}_0:=\arg\max_{\|\bm{z}\|=1}~\,\bm{z}^\ast\Big(\frac{1}{|\mathcal{S}|}\sum_{i\in\mathcal{S}}\psi_i^\gamma\bm{a}_i\bm{a}_i^\ast\Big)\bm{z}
\end{equation}
which can be efficiently evaluated in time proportional to $\mathcal{O}(n|\mathcal{S}|)$ by means of the power method or the Lanczos algorithm \cite{book2011saad}.  The proposed 
initialization can be obtained upon scaling $\tilde{\bm{z}}_0$ from \eqref{eq:rmci} by the norm estimate, to yield $\bm{z}_0:=(\sum_{i=1}^m\!\nicefrac{\psi_i^2}{m})\tilde{\bm{z}}_0$. 
By default, we take $\gamma:=\nicefrac{1}{2}$ in all reported numerical implementations, yielding weights $w_{i}^0:=\sqrt{|\langle\bm{a}_i,\bm{x}\rangle|}$ for all $i\in\mathcal{S}$. 

Regarding the initialization procedure in \eqref{eq:rmci}, we next highlight two features, while details and theoretical performance guarantees are provided in Section \ref{sec:main}:
\begin{enumerate}
	\item[\textbf{F1)}] 
	The weights $\{w_{i}^0\}$ in the maximal correlation scheme enable leveraging useful information that each feature vector $\bm{a}_i$ may bear regarding the direction of $\bm{x}$. 
	\item[\textbf{F2)}] 
	Taking $w_i^0:=\psi_i^\gamma$ for all $i\in\mathcal{S}$ and $0$ otherwise, problem \eqref{eq:rmci} can be equivalently rewritten as
	\begin{equation}
	\tilde{\bm{z}}_0:=\arg\max_{\|\bm{z}\|=1}~~\bm{z}^\ast\Big(\frac{1}{m}\sum_{i=1}^mw_i^0\bm{a}_i\bm{a}_i^\ast\Big)\bm{z}
	\end{equation} 
	which subsumes existing initialization schemes with particular selections of weights. For instance, the ``plain-vanilla'' spectral initialization in \cite{altmin,wf} is recovered by choosing $\mathcal{S}:=\mathcal{M}$, and $w_i^0:=\psi_i^2$ for all $1\le i\le m$. 
\end{enumerate}

\begin{figure}[t]
	\vspace{-10pt}
	\centering
	\includegraphics[scale = 0.58]{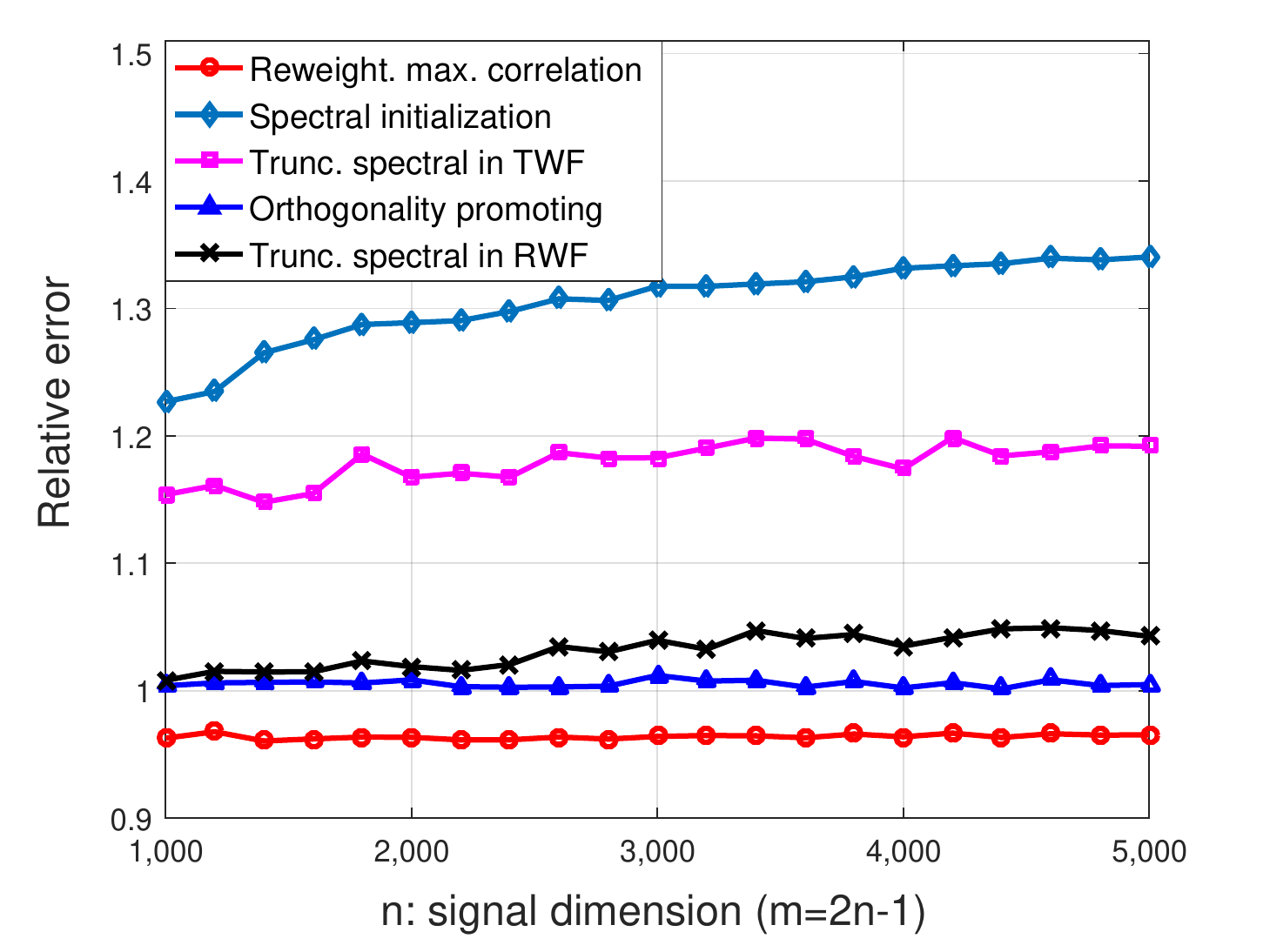}
	\vspace{-0pt}
	\caption{Relative initialization error for the real-valued Gaussian model with $n=1,000$ and $m=2n-1=1,9999$.}
	\label{fig:initratio}
	\vspace{-0pt}
\end{figure}
For numerical comparison, define the 
$	{\rm Relative~error}:=\nicefrac{{\rm dist}(\bm{z},\bm{x})}{\|\bm{x}\|}$.
All simulated results reported in this paper were averaged over $100$ Monte Carlo (MC) realizations. 
Figure \ref{fig:initratio} evaluates the performance of the proposed initialization relative to several state-of-the-art strategies, and also with the information limit number of data benchmarking the minimal number of samples required.
It is clear that our initialization is: i) consistently better than the state-of-the-art;
and, ii) stable as the signal dimension $n$ grows, which is in sharp contrast to the instability encountered by the spectral ones \cite{altmin,wf,twf,reshaped}. 
It is also worth stressing that the about $5\%$ empirical advantage (over the best \cite{taf}) at the challenging \emph{information-theoretic benchmark} is indeed nontrivial, and is one of the main RAF upshots. This numerical advantage becomes increasingly pronounced as the ratio $m/n$ of the number of equations to the unknowns grows. In this regard, our proposed initialization procedure may be combined with other iterative phase retrieval approaches to improve their numerical performance.

\subsection{Adaptively reweighted gradient flow}
\label{subsec:gradient}

For independent data adhering to the real-valued Gaussian model, the direction that TAF moves along in stage S2) presented earlier is given by the following (generalized) gradient \cite{taf}, \cite{clarke1975gg}:
\begin{equation}
\label{eq:grad}
\frac{1}{m}\sum_{i\in\mathcal{T}}\nabla \ell(\bm{z};\psi_i)=\frac{1}{m}\sum_{i\in\mathcal{T}}\Big(\bm{a}_i^\ast\bm{z}-\psi_i\frac{\bm{a}_i^\ast\bm{z}}{|\bm{a}_i^\ast\bm{z}|}\Big)\bm{a}_i
\end{equation}
where the dependence on the iterate count $t$ is neglected for notational brevity, and the convention $\nicefrac{\bm{a}_i^\ast\bm{z}}{|\bm{a}_i^\ast\bm{z}|}:=0$ is adopted if $\bm{a}_i^\ast\bm{z}=0$. 

Unfortunately, the (negative) gradient of the average in \eqref{eq:grad} may not point towards the true $\bm{x}$ unless the current iterate $\bm{z}$ is already very close to $\bm{x}$. As a consequence, moving along such a descent direction may not drag $\bm{z}$ closer to $\bm{x}$. To see this, consider an initial guess $\bm{z}_0$ that has already been in a \emph{basin of attraction} (i.e., a region within which there is only a unique stationary point) of $\bm{x}$. Certainly, there are summands $(\bm{a}_i^\ast\bm{z}-\psi_i\nicefrac{\bm{a}_i^\ast\bm{z}}{|\bm{a}_i^\ast\bm{z}|})\bm{a}_i$ in \eqref{eq:grad}, that could give rise to ``bad/misleading'' search directions due to the erroneously estimated signs $\nicefrac{\bm{a}_i^\ast\bm{z}}{|\bm{a}_i^\ast\bm{z}|}\ne \nicefrac{\bm{a}_i^\ast\bm{x}}{|\bm{a}_i^\ast\bm{x}|}$  in \eqref{eq:grad} \cite{taf}. Those gradients as a whole may drag $\bm{z}$ away from $\bm{x}$, and hence out of the basin of attraction. Such an effect becomes increasingly severe as the number $m$ of acquired examples approaches the information-theoretic limit of $2n-1$, thus rendering past approaches less effective in this case. Although this issue is somewhat remedied by TAF with a truncation procedure, its efficacy is still limited due to misses of bad gradients and mis-rejections of meaningful ones around the information-theoretic limit.

To address this challenge, our reweighted gradient flow effecting suitable search directions from \emph{almost all} acquired data samples $\{(\bm{a}_i;\psi_i)\}_{1\le i\le m}$ will be adopted in a (timely) adaptive fashion; that is,
\begin{equation}
\label{eq:raf}
\bm{z}^{t+1}=\bm{z}^t-{\mu^t}\nabla\ell_{\rm rw}(\bm{z}^t;\psi_i),\quad t=0,\,1,\,\ldots
\end{equation}
The \emph{reweighted gradient} $\nabla\ell_{\rm rw}(\bm{z}^t)$ evaluated at the current point $\bm{z}^t$ is given as
\begin{equation}
\label{eq:rgrad}
\nabla\ell_{\rm rw}(\bm{z})\!:=\!\frac{1}{m}	\sum_{i=1}^m w_i\nabla\ell(\bm{z};\psi_i)
\end{equation}
for suitable weights $\{w_i\}_{1\le i\le m}$ to be designed shortly.

To that end, we observe that the truncation criterion
$\mathcal{T}\!:=\!\{1\le i\le m:\nicefrac{|\bm{a}_i^\ast\bm{z}|}{|\bm{a}_i^\ast\bm{x}|}\ge \alpha\}$ with some given parameter $\alpha>0$ suggests to include only gradients associated with $|\bm{a}_i^\ast\bm{z}|$ of relatively large sizes. 
This is because gradients of sizable $\nicefrac{|\bm{a}_i^\ast\bm{z}|}{|\bm{a}_i^\ast\bm{x}|}$ offer reliable and meaningful directions pointing to the truth $\bm{x}$ with large probability \cite{taf}. As such, the ratio $\nicefrac{|\bm{a}_i^\ast\bm{z}|}{|\bm{a}_i^\ast\bm{x}|}$ can be somewhat viewed as a confidence score about the reliability or meaningfulness of the corresponding gradient $\nabla\ell(\bm{z};\psi_i)$.
Recognizing that confidence can vary, it is natural to distinguish the contributions that different gradients make to the overall search direction.  
An easy way is to attach large weights to the reliable gradients, and small weights to the spurious ones. Assume without loss of generality that $0\le w_i\le 1$ for all $1\le i\le m$; otherwise, lump the normalization factor achieving this into the learning rate $\mu^t$. Building upon this observation and leveraging the gradient reliability confidence score $\nicefrac{|\bm{a}_i^\ast\bm{z}|}{|\bm{a}_i^\ast\bm{x}|}$, the weight per gradient $\nabla\ell(\bm{z};\psi_i)$ in our proposed RAF algorithm is designed to be 
\begin{equation}
\label{eq:weights}
w_i:=\frac{1}{1+\nicefrac{\beta_i}{
		(\nicefrac{|\bm{a}_i^\ast\bm{z}|}{|\bm{a}_i^\ast\bm{x}|})}}
,\quad 1\le i\le m
\end{equation}
where $\{\beta_i>0\}_{1\le i\le m}$ are some pre-selected parameters. 

Regarding the weighting criterion in \eqref{eq:weights}, three remarks are in order:

\begin{enumerate}
	\item[\textbf{R1)}] The weights $\{w_i^t\}_{1\le i\le m}$ are time adapted to the iterate $\bm{z}^t$. One can also interpret the reweighted gradient flow $\bm{z}^{t+1}$ in \eqref{eq:raf} as performing a single gradient step to minimize the \emph{smooth} \emph{reweighted} loss $\frac{1}{m}\sum_{i=1}^mw_i^t\ell(\bm{z};\psi_i)$ with starting point $\bm{z}^t$;
	see also \cite{irls2008}
	 for related ideas successfully exploited in the \emph{iteratively reweighted least-squares} approach to compressive sampling.
	
	\item[\textbf{R2)}] Note that the larger the confidence score $\nicefrac{|\bm{a}_i^\ast\bm{z}|}{|\bm{a}_i^\ast\bm{x}|}$ is, the larger the corresponding weight $w_i$ will be. More importance will be then attached to reliable gradients than to spurious ones. Gradients from \emph{almost all} data are are accounted for, which is in contrast to \cite{taf}, where withdrawn gradients do not contribute the information they carry. 
	
	\item[\textbf{R3)}] At the points $\{\bm{z}\}$ where $\bm{a}_i^\ast\bm{z}=0$ for some datum $i\in\mathcal{M}$, the $i$-th weight will be $w_i=0$. In other words, the squared losses $\ell(\bm{z};\psi_i)$ in \eqref{eq:ls} that are nonsmooth at points $\bm{z}$ will be eliminated, to prevent their contribution to the reweighted gradient update in \eqref{eq:raf}. Hence, the convergence analysis of RAF is considerably simplified because it does not have to cope with the nonsmoothness of the objective function in \eqref{eq:ls}.
	
\end{enumerate}

Having elaborated on the two stages, RAF can be readily summarized in Algorithm~\ref{alg:raf}. 

\subsection{Algorithmic parameters}
To optimize the empirical performance and facilitate numerical implementations, choice of pertinent algorithmic parameters of RAF is independently discussed here. It is obvious that the RAF algorithm entails four parameters. Our theory and all experiments are based on: i) $\nicefrac{|\mathcal{S}|}{m}\le 0.25$; ii) 
$0\le \beta_i\le 10$ for all $1\le i\le m$; and, iii) $0\le \gamma\le 1$.   
For convenience, a constant step size $\mu^t\equiv\mu>0$ is suggested, but other step size rules such as backtracking line search with the reweighted objective work as well. As will be formalized in Section \ref{sec:main}, RAF converges if the constant $\mu$ is not too large, with the upper bound depending in part on the selection of $\{\beta_i\}_{1\le i\le m}$.

In the numerical tests presented in Section \ref{sec:alg} and \ref{sec:test}, we take $|\mathcal{S}|:=\lfloor\nicefrac{3m}{13}\rfloor $, $\beta_i\equiv\beta: =10$, $\gamma:=0.5$, and $\mu:=2$ (larger step sizes can be afforded for larger $m/n$ values).

\begin{algorithm}[t]
	\caption{Reweighted Amplitude Flow}
	\label{alg:raf}
	\begin{algorithmic}[1]
		\STATE {\bfseries Input:}
		Data $\{(\bm{a}_i;\psi_i\}_{1\le i\le m}$; maximum number of iterations $T$; step sizes $\mu^t=2/6$ and weighting parameters $\beta_i=10/5$ for real-/complex-valued Gaussian models; subset cardinality $|\mathcal{S}|=\lfloor\nicefrac{3m}{13}\rfloor$, and exponent $\gamma=0.5$.
		\STATE {\bfseries Construct} $\mathcal{S}$ to include indices associated with the $|\mathcal{S}|$ largest entries among $\{\psi_i\}_{1\le i\le m}$.
		\STATE {\bfseries Initialize} \label{step:3} $\bm{z}^0:=\sqrt{\nicefrac{\sum_{i=1}^m\psi_i^2}{m}}\,\tilde{\bm{z}}^0$ with $\tilde{\bm{z}}^0$ being 
		the unit principal eigenvector of 
		\begin{equation}
		\label{eq:ymatrix}
		\bm{Y}:=\frac{1}{m}\sum_{i=1}^mw_i^0 \bm{a}_i\bm{a}_i^\ast,\quad {\rm where}\quad w_i^0:=\left\{\!\!\begin{array}{cc}
		\psi_i^\gamma,& i\in\mathcal{S}\!\subseteq\!\mathcal{M}\!\\
		0,&{\rm otherwise}
		\end{array}
		\right..
		\end{equation}
		\STATE {\bfseries Loop: for} \label{step:4}
		{$t=0$ {\bfseries to} $T-1$}
		{ \begin{equation}
			\bm{z}^{t+1}=\bm{z}^t-\frac{\mu^t}{m}	\sum_{i=1}^m w_i^t\,
			\Big(\bm{a}_i^\ast\bm{z}^t-\psi_i\frac{\bm{a}_i^\ast\bm{z}^t}{|\bm{a}_i^\ast\bm{z}^t|}\Big)\bm{a}_i
			\label{eq:final}
			\end{equation}  
			where $w_i^t:=\frac{\nicefrac{|\bm{a}_i^\ast\bm{z}^t|}{\psi_i}}{\nicefrac{|\bm{a}_i^\ast\bm{z}^t|}{\psi_i}+\beta_i}$ for all $1\le i \le m$. }
		\STATE {\bfseries Output:}\label{step:5} 
		$\bm{z}^{T}$.
	\end{algorithmic}
\end{algorithm} 


\vspace{-3pt}
\section{Main Results}
\label{sec:main}

Our main result summarized in Theorem \ref{thm:initial} next establishes exact recovery under the real-valued Gaussian model, whose proof is postponed to Section \ref{sec:proof} for readability. Our RAF methogolody however can be generalized readily to the complex Gaussian and CDP models.   

\begin{theorem}[Exact recovery]
	\label{thm:initial}
	Consider $m$ noiseless measurements $\bm{\psi}=|\bm{A}\bm{x}|$ for
	an arbitrary signal $\bm{x}\in\mathbb{R}^n$.
	If $m\ge c_0|\mathcal{S}|\ge c_1 n$ with $|\mathcal{S}|$ being the pre-selected subset cardinality in the initialization step
	and the learning rate $\mu\le \mu_0$, then with probability at least $1-c_3{\rm e}^{-c_2m}$, the reweighted amplitude flow's estimates $\bm{z}^t$ in Algorithm~\ref{alg:raf} 
	obey 
	\begin{equation}\label{eq:linear}
	{\rm dist}( \bm{z}^t,\bm{x})\le \frac{1}{10}(1-\nu)^t\|\bm{x}\|,\quad t=0,1,\ldots
	\end{equation} 
	where $c_0,\,c_1,\,c_2,\,c_3>0$, 
	$0<\nu<1$, and $\mu_0>0$ are certain numerical constants depending on the choice of algorithmic parameters $|\mathcal{S}|,\,\beta$, $\gamma$, and $\mu$.  
\end{theorem}

According to Theorem \ref{thm:initial}, a few interesting properties of our RAF algorithm are worth highlighting. 
To start, RAF recovers the true solution exactly with high probability whenever the ratio $m/n$ of the number of equations to the unknowns exceeds some numerical constant. Expressed differently, RAF achieves the information-theoretic optimal order of sample complexity, which is consistent with the state-of-the-art including truncated Wirtinger flow (TWF) \cite{twf}, TAF \cite{taf}, and RWF \cite{reshaped}. Notice that the error contraction in \eqref{eq:linear} also holds at $t=0$, namely, ${\rm dist}(\bm{z}^0,\bm{x})\le\nicefrac{ \|\bm{x}\|}{10}$, therefore providing theoretical performance guarantees for the proposed initialization strategy (cf. Step \ref{step:3} of Algorithm \ref{alg:raf}).  
Moreover, starting from this initial estimate, RAF converges exponentially fast to the true solution $\bm{x}$. In other words, to reach any $\epsilon$-relative solution accuracy (i.e., ${\rm dist}(\bm{z}^T,\bm{x})\le \epsilon\|\bm{x}\|$), it suffices to run at most $T=\mathcal{O}(\log\nicefrac{1}{\epsilon})$ RAF iterations in Step \ref{step:4} of Algorithm \ref{alg:raf}. This in conjunction with the per-iteration complexity $\mathcal{O}(mn)$ (namely, the complexity of one reweighted gradient update in \eqref{eq:final}) confirms that RAF solves exactly a quadratic system in time $\mathcal{O}(mn\log \nicefrac{1}{\epsilon})$, which is linear in $\mathcal{O}(mn)$, the time required by the processor to read the entire data $\{(\bm{a}_i;\psi_i)\}_{1\le i\le m}$. Given the fact that the initialization stage can be performed in time $\mathcal{O}(n|\mathcal{S}|)$ and $|\mathcal{S}|< m$, 
the overall linear-time complexity of RAF is order-optimal.     

\section{Simulated Tests}\label{sec:test}

\begin{figure}[t]
	\vspace{-20pt}
	\centering
	\includegraphics[width=0.7\columnwidth, height=7.7cm]{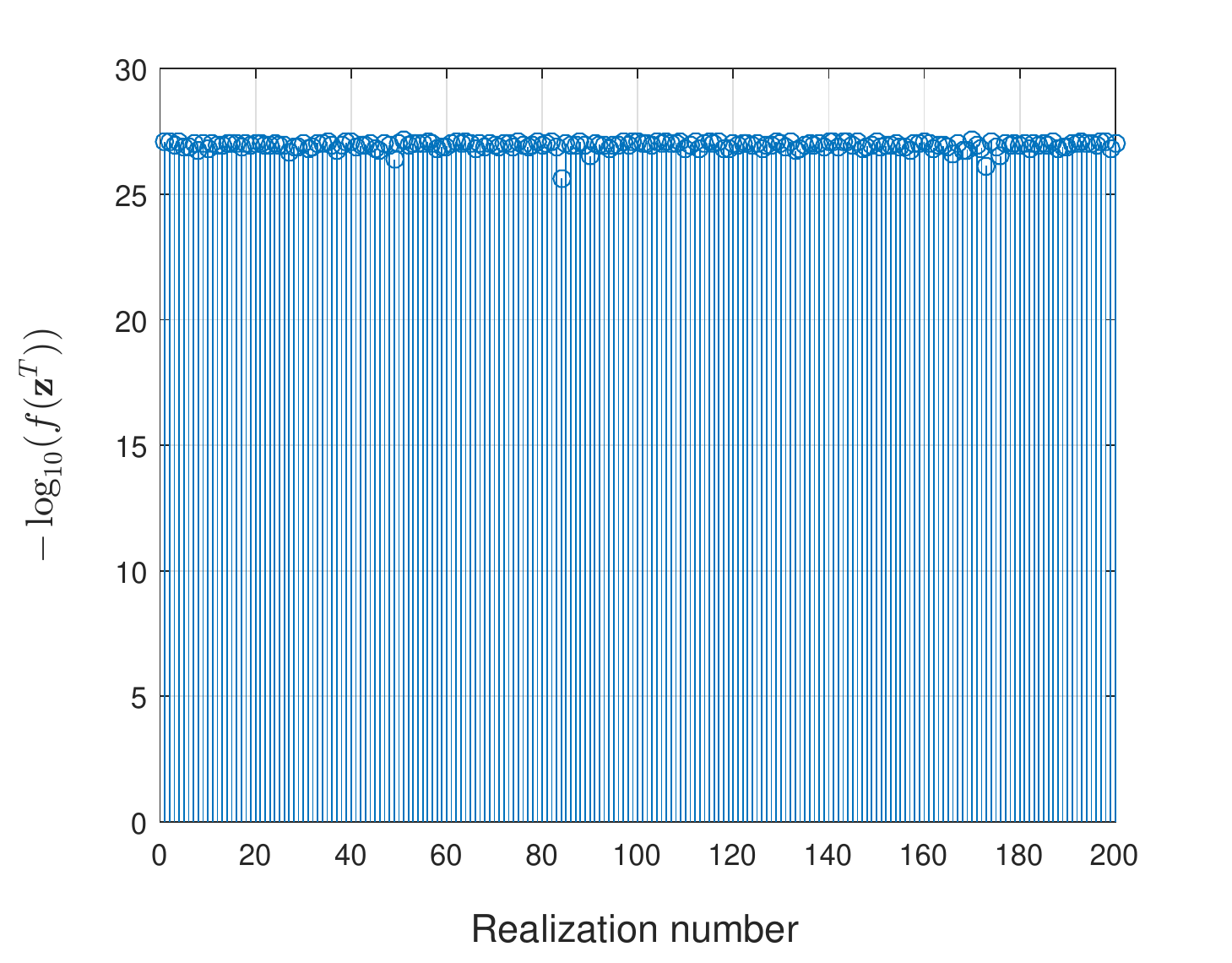}
	\caption{Function value $L(\bm{z}^T)$ evaluated at the returned RAF estimate $\bm{z}^T$ for $200$ trials with $n=2,000$ and $m=2n-1=3,999$.}
	\label{fig:func}
	\vspace{-0pt}
\end{figure}
Our theoretical findings about RAF have been corroborated with comprehensive numerical experiments, a sample of which are discussed next. 
Performance of RAF is evaluated relative to the state-of-the-art (T)WF \cite{wf,twf}, RWF \cite{reshaped}, and TAF \cite{taf} in terms of the empirical success rate among $100$ MC realizations,  
where a success will be declared for an independent trial if the returned estimate incurs error 
$\nicefrac{\|\bm{\psi}-|\bm{A}\bm{z}^T|\|}{\|\bm{x}\|}$
less than $10^{-5}$. 
Both the real Gaussian and the physically realizable CDP models were simulated. 
For fairness, all procedures were implemented with their suggested parameter values. 
We generated the truth $\bm{x}\sim\mathcal{N}(\bm{0},\bm{I})$, and i.i.d. measurement vectors $\bm{a}_i\sim\mathcal{N}(\bm{0},\bm{I})$, $1\le i\le m$. 
Each iterative scheme 
obtained its initial guess based on $200$ power or Lanczos iterations, followed by a sequence of $T=2,000$ (which can be set smaller as the ratio $m/n$ grows away from the limit of $2$) gradient-type iterations. 
For reproducibility, the Matlab implementation of our RAF algorithm is publicly available at \url{https://gangumn.github.io/RAF/}.

To show the power of RAF in the high-dimensional regime, the function value $L(\bm{z})$ in \eqref{eq:ls} evaluated at the returned estimate $\bm{z}^T$ (cf. Step \ref{step:5} of Algorithm \ref{alg:raf}) for $200$ MC realizations
is plotted (in negative logarithmic scale) in Fig. \ref{fig:func}, where the number of simulated noiseless measurements was set to be the information-theoretic limit, namely, 
$m=2n-1=3,999$ for $n=2,000$. It is self-evident that our proposed RAF approach returns a solution of function value $L(\bm{z}^T)$ smaller than $10^{-25}$ in all $200$ independent realizations 
even at this challenging information-theoretic limit condition. To the best of our knowledge, RAF is the first algorithm that empirically reconstructs any high-dimensional (say e.g., $n\ge 1,500$) signals exactly from an \emph{optimal number} of random quadratic equations, which also provides a positive answer to the question posed easier in the Introduction. 

The left panel in Fig.~\ref{fig:ireal} further
compares the empirical success rate of five schemes 
with the signal dimension being fixed at $n=1,000$ while the ratio
$m/n$ increasing by $0.1$ from $1$ to $5$. As clearly depicted by the plots, our RAF (the red plot) enjoys markedly improved performance over its competing alternatives. Moreover, it also achieves $100\%$ perfect signal recovery as soon as $m$ is about $2n$, where the others do not work (well).
To numerically demonstrate the stability and robustness of RAF in the presence of additive noise, the right panel in Fig.~\ref{fig:ireal} examines the normalized mean-square error  ${\rm NMSE}:=\nicefrac{{\rm dist}^2(\bm{z}^T,\bm{x})}{\|\bm{x}\|^2}$ as a function of the signal-to-noise ratio (SNR) for $m/n$ taking values $\{3,4,5\}$. The noise model $\psi_i=|\langle\bm{a}_i,\bm{x}\rangle |+\eta_i$ with $\bm{\eta}:=[\eta_i]_{1\le i\le m}\sim \mathcal{N}(\bm{0},\sigma^2\bm{I}_m)$ was simulated, where $\sigma^2$ was set such that certain ${\rm SNR}:=10\log_{10}(\nicefrac{\|\bm{A}\bm{x}\|^2}{m\sigma^2})$ values were achieved. For all choices of $m$ (as small as $3n$ which is nearly minimal), the numerical experiments illustrate that the NMSE scales inversely proportional to the SNR, which corroborates the stability of our RAF approach.

\begin{figure}[t]
	\centering
	\begin{subfigure}
		\centering
		\includegraphics[width=.50\columnwidth, height=5.9cm]{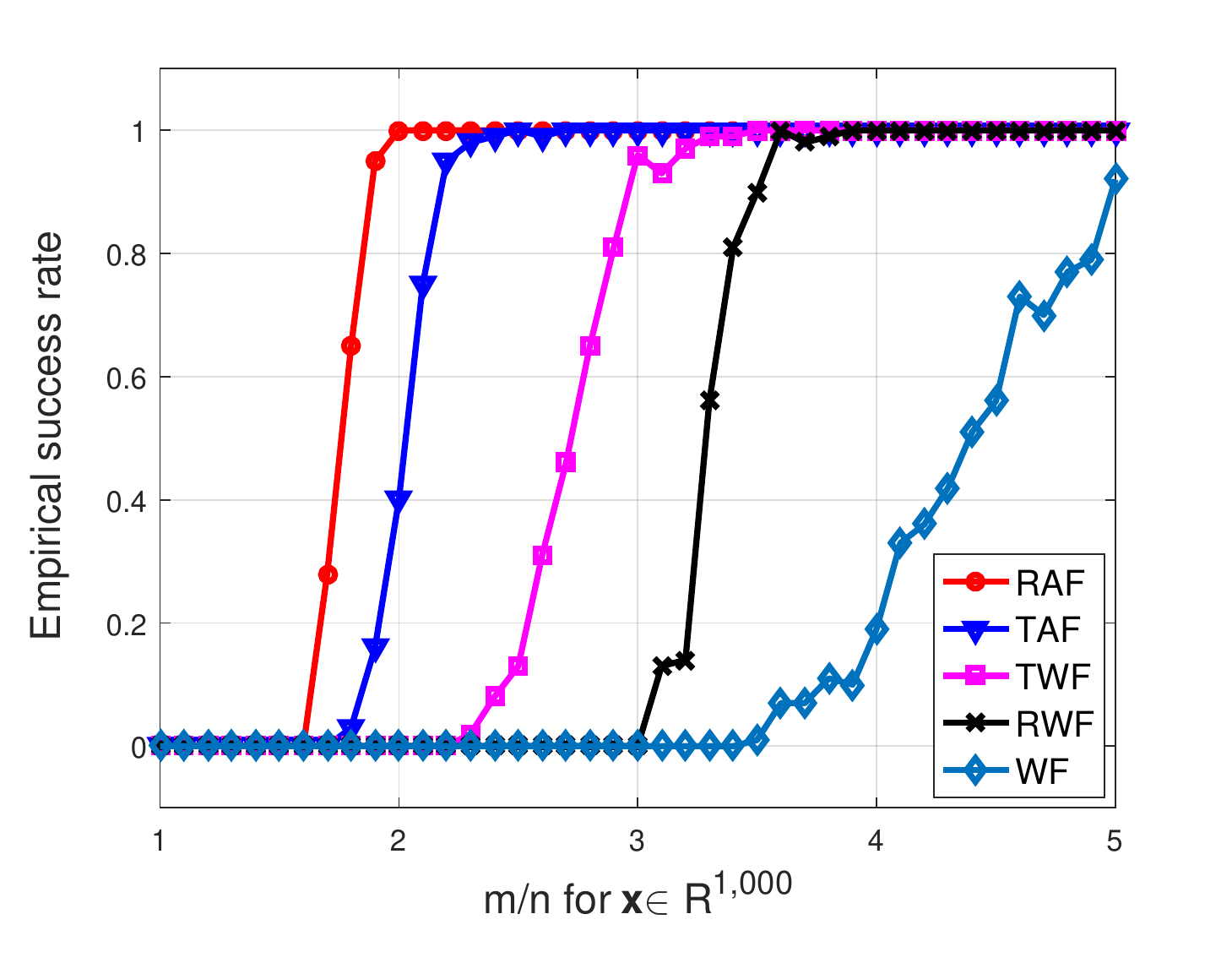}
	\end{subfigure}
	\hspace{-15pt}
	\begin{subfigure}
		\centering
		\includegraphics[width=.51\columnwidth, height=5.8cm]{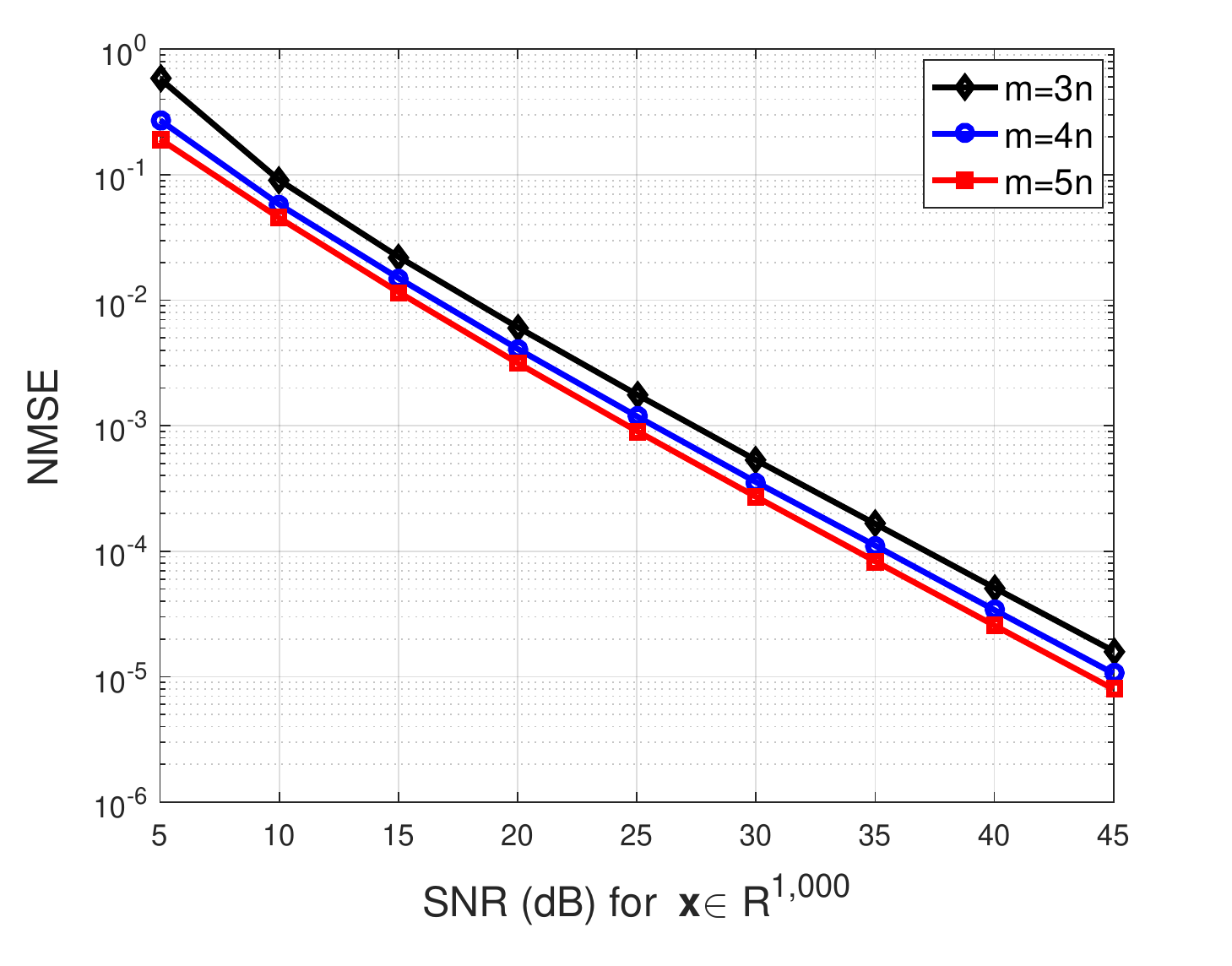}
	\end{subfigure}
	\vspace{-5pt}
	\caption{ Real-valued Gaussian model with $\bm{x}\in\mathbb{R}^{1,000}$: Empirical success rate (Left); and, NMSE vs. SNR (Right).
	}
	\label{fig:ireal}
	\vspace{-5pt}
\end{figure}

To demonstrate the efficacy and scalability of RAF in real-world conditions, the last experiment entails the Galaxy image \footnotemark\footnotetext{Downloaded from \url{http://pics-about-space.com/milky-way-galaxy}.} depicted by a three-way array $\bm{X}\in\mathbb{R}^{1,080\times 1,920\times 3}$, whose first two coordinates encode the pixel locations, and the third the RGB color bands. 
Consider the physically realizable CDP model with random masks~\cite{wf}.
Letting $\bm{x}\in\mathbb{R}^n$ ($n\approx 2\times 10^6$) be a vectorization of a certain band of $\bm{X}$, the CDP model with $K$ masks is
$ \bm{\psi}^{(k)}=|\bm{F}\bm{D}^{(k)}\bm{x}|,~1\le k\le K$,
where $\bm{F}\in\mathbb{C}^{n\times n}$ is a discrete Fourier transform matrix, and diagonal matrices $\bm{D}^{(k)}$ have their diagonal entries sampled uniformly at random from $\{1,-1,j,-j\}$ with $j:=\sqrt{-1}$. 
Implementing $K=4$ masks, each algorithm performs independently over each band $100$ power iterations for an initial guess, which was refined by $100$ gradient iterations.   
Recovered images of TAF (left) and RAF (right) are displayed in Fig.~\ref{fig:milky}, whose relative errors were $1.0347$ and $1.0715\times 10^{-3}$, respectively. WF and TWF returned images of corresponding relative error $1.6870$ and $1.4211$, which are far away from the ground truth.  

Regarding running times, RAF converges faster both in time and in the number of iterations required to achieve certain solution accuracy than TWF and WF in all simulated experiments, and it has comparable efficiency as TAF and RWF. All numerical experiments were implemented with MATLAB R$2016$a on an Intel CPU @ $3.4$ GHz ($32$ GB RAM) computer.



	\begin{figure}[ht]
		\centering
		\vspace{-15pt}
		\includegraphics[height=.24\textheight, width=1\textwidth]{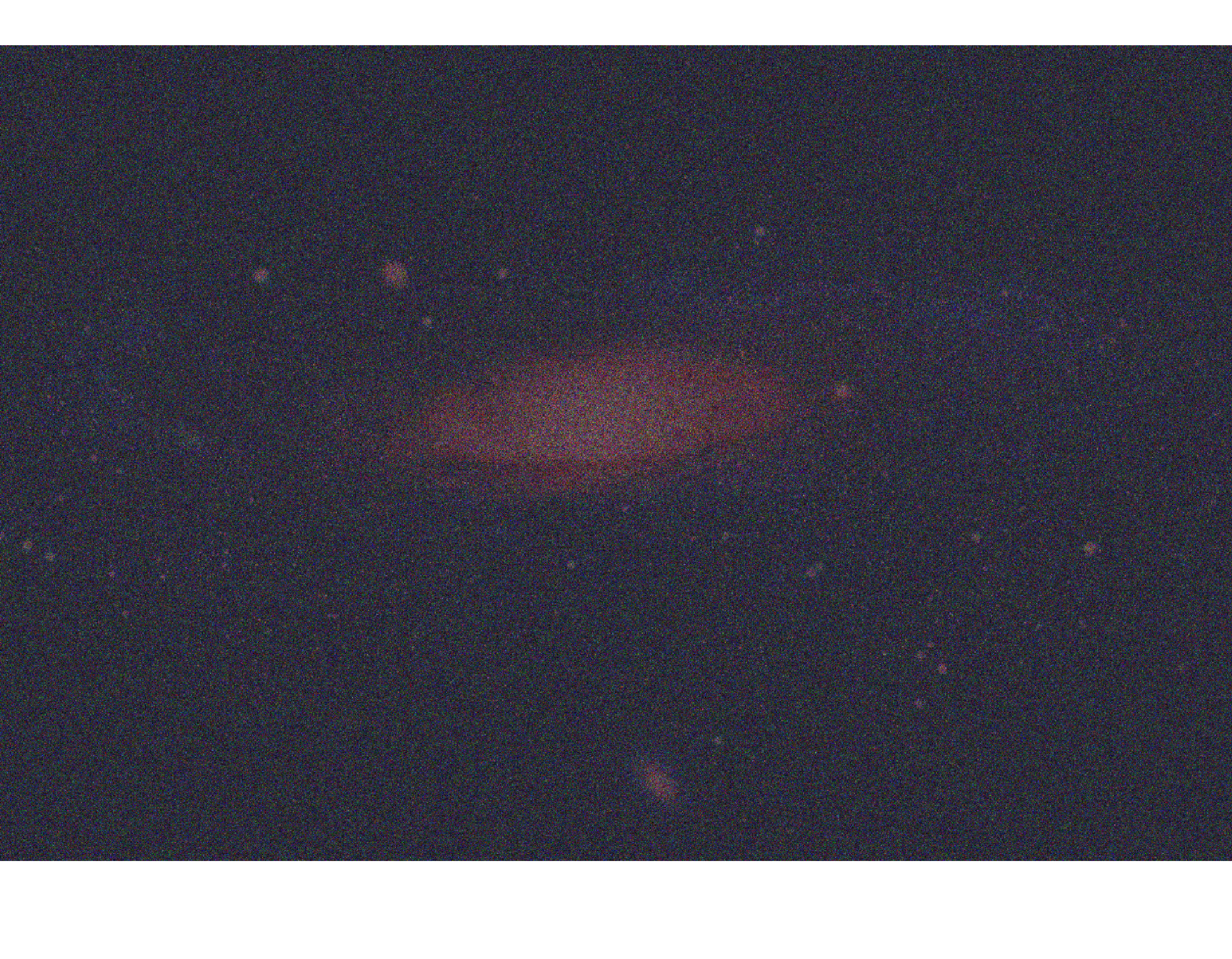} 
		\vspace{-5pt}
		\includegraphics[height=.24\textheight, width=1\textwidth]{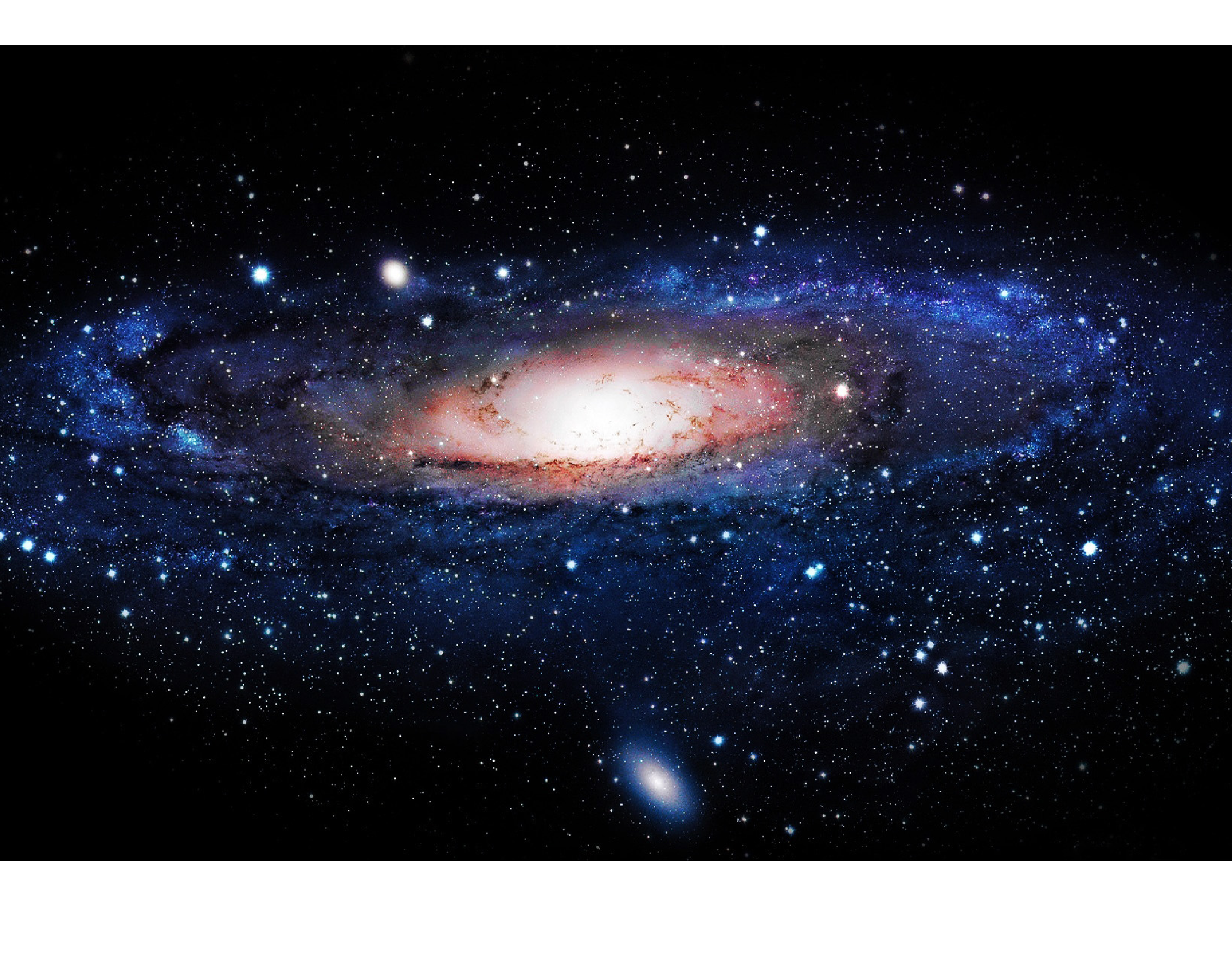}
		\caption{
			The recovered Milky Way Galaxy images after $100$ truncated gradient iterations of TAF (Top); and after $100$ reweighted gradient iterations of RAF (Bottom).
		}
		\label{fig:milky}
		\vspace{-0pt}
	\end{figure}

\section{Proofs}\label{sec:proof}


To prove Theorem~\ref{thm:initial}, 
this section establishes a few lemmas and the main ideas, whereas technical details are postponed to the Appendix for readability. 
 It is clear from Algorithm \ref{alg:raf}
 that the weighted maximal correlation initialization (cf. Step \ref{step:3}) and the reweighted gradient flow (cf. Step \ref{step:4}) distinguish themselves from those procedures in
 (T)WF \cite{wf,twf}, TAF \cite{taf}, and RWF \cite{reshaped}.
Hence, new proof techniques to cope with the weighting in both the initialization and the gradient flow, as well as the nonsmoothness and nonconvexity of the amplitude-based least-squares functional are required. Nevertheless, part of the proof is built upon those in \cite{wf}, \cite{reshaped}, \cite{taf}, \cite{2015chen1}.

The proof of Theorem~\ref{thm:initial} is based on two parts: Section~\ref{subsec:initial} below corroborates guaranteed theoretical performance of the proposed initialization, which essentially achieves any given constant relative error as soon as the number of equations is on the order of the number of unknowns; that is, $m \ge c_1 n$ for some constant $c_1>0$. It is worth mentioning that we reserve $c$ and its subscripted versions for absolute constants, and their values may vary with the context.
Under the sample complexity of order $\mathcal{O}(n)$, Section~\ref{subsec:er} further shows that RAF converges to the true signal $\bm{x}$ exponentially fast whenever the initial estimate lands within a relatively small-size neighborhood of $\bm{x}$ defined by ${\rm dist}(\bm{z}^0,\bm{x})\le (\nicefrac{1}{10}) \|\bm{x}\|$. 


\subsection{Weighted maximal correlation initialization}\label{subsec:initial}

This section is devoted to developing theoretical guarantees for the novel initialization procedure, which is summarized in the following proposition.

\begin{proposition}\label{prop:initial}
	For arbitrary $\bm{x}\in\mathbb{R}^n$, consider the noiseless measurements $\psi_i=|\bm{a}_i^\ast\bm{x}|$, $1\le  i\le m$. 
	If $m\ge c_0|\mathcal{S}|\ge c_1n$, then with probability exceeding $1-c_3{\rm e}^{-c_2m}$, the initial guess $\bm{z}^0$ obtained by the weighted maximal correlation method in Step \ref{step:3} of Algorithm \ref{alg:raf} satisfies
	\begin{equation}\label{eq:i1x}
	{\rm dist}(\bm{z}^0,\bm{x})\le \rho\|\bm{x}\|
	\end{equation}
	for $\rho=\nicefrac{1}{10}$ (or any sufficiently small positive number).		
	Here, $c_0,\,c_1,\,c_2,\,c_3>0$ are some absolute constants.
\end{proposition}

Due to the homogeneity of ~\eqref{eq:i1x}, it suffices to prove the result for the case of $\|\bm{x}\|=1$. Assume first that the norm $\|\bm{x}\|=1$ is also perfectly known, and 
$\bm{z}^0$ has already been scaled such that $\|\bm{z}^0\|=1$. 
At the end of this proof, this approximation error between the actually employed norm estimate $\sqrt{\nicefrac{\sum_{i=1}^m y_i}{m}}$ found based on the strong law of large numbers
and the unknown norm $\|\bm{x}\|=1$ will be taken care of.
For independent Gaussian random measurement vectors $\bm{a}_i\sim\mathcal{N}(\bm{0},\bm{I}_n)$ and arbitrary unit signal vector $\bm{x}$,
there always exists an orthogonal transformation denoted by $\bm{U}\in\mathbb{R}^{n\times n}$ such that $\bm{x}=\bm{U}\bm{e}_1$. Since
\begin{equation}
|\langle\bm{a}_i,\bm{x}\rangle|^2=|\langle\bm{a}_i,\bm{U}\bm{e}_1\rangle|^2=|\langle\bm{U}^\ast\bm{a}_i,\bm{e}_1\rangle|^2\eqdef |\langle\bm{a}_i,\bm{e}_1\rangle|^2
\end{equation}
where $\eqdef$ means random quantities on both sides of the equality have the same distribution, it is thus without loss of generality to work with $\bm{x}=\bm{e}_1$.

Since the norm $\|\bm{x}\|=1$ is assumed known, the weighted maximal correlation initialization in Step \ref{step:3} finds the initial estimate $\bm{z}^0=\tilde{\bm{z}}^0$ (the scaling factor is the exactly known norm $1$ in this case) as the principal eigenvector of 
\begin{equation}
\label{eq:eig}
\bm{Y}:=\frac{1}{|\mathcal{S}|}\bm{B}^\ast\bm{B}=\frac{1}{|\mathcal{S}|}\sum_{i\in\mathcal{S}}\psi_i^\gamma\bm{a}_i\bm{a}_i^\ast
\end{equation}
where $\bm{B}:=\big[\psi_i^{\nicefrac{\gamma}{2}}\bm{a}_i\big]_{i\in\mathcal{S}}$ is an $|\mathcal{S}|\times n$ matrix, and $\mathcal{S}\subsetneqq\{1,2,\ldots,m\}$ includes the indices of the $|\mathcal{S}|$ largest entities among all modulus data $\{\psi_i\}_{1\le i\le m}$.
The following result is a modification of \cite[Lemma 1]{taf}, which is key to proving Proposition~\ref{prop:initial} and 
whose proof can be found in Section \ref{sec:lembeta} in the Appendix. 

\begin{lemma}\label{lem:beta}
	Consider $m$ noiseless measurements $\psi_i=|\bm{a}_i^\ast\bm{x}|$, $1\le  i\le m$. For arbitrary $\bm{x}\in\mathbb{R}^n$ of unity norm, the next result holds for all unit vectors $\bm{u}\in\mathbb{R}^n$ perpendicular to the vector $\bm{x}$; that is, for all vectors $\bm{u}\in\mathbb{R}^n$ satisfying $\bm{u}^\ast\bm{x}=0$ and $\|\bm{u}\|=1$:
	\begin{equation}\label{eq:mse}
	\frac{1}{2}\|\bm{x}\bm{x}^\ast-\bm{z}^0(\bm{z}^0)^\ast\|^2_F\le \frac{\|\bm{B}\bm{u}\|^2}{\|\bm{B}\bm{x}\|^2}
	\end{equation}
	where $\bm{z}^0=\tilde{\bm{z}}^0$ is given by
	\begin{equation}
	\label{eq:maxeig}
	\tilde{\bm{z}}^0:=\arg\max_{\|\bm{z}\|=1}~\frac{1}{|\mathcal{S}|}\bm{z}^\ast\bm{B}^\ast\bm{B}\bm{z}.
	\end{equation}
\end{lemma}

In the sequel, we start proving Proposition~\ref{prop:initial}. 
The first step consists in upper-bounding the quantity on the right-hand-side of~\eqref{eq:mse}. To be specific, this task involves upper bounding its numerator, and lower bounding its denominator, which are summarized in Lemma~\ref{lem:up} and Lemma~\ref{lem:low} and whose proofs are deferred to Section~\ref{sec:proofup} and Section \ref{sec:prooflow} in the Appendix, accordingly.

\begin{lemma}\label{lem:up}
	In the setting of Lemma~\ref{lem:beta}, if $\nicefrac{|\mathcal{S}|}{n}\ge c_4$, then the next 
	\begin{equation}\label{eq:up0}
	\|\bm{B}\bm{u}\|^2\le 1.01\sqrt{\nicefrac{2^\gamma}{\pi}} \Gamma(\nicefrac{\gamma+1}{2})|\mathcal{S}|
	\end{equation}
	holds with probability at least $1-2{\rm e}^{-c_5 n}$, where $\Gamma(\cdot)$ is the Gamma function, and
	 $c_4,\,c_5$ are certain universal constants.
\end{lemma}

\begin{lemma}\label{lem:low}
	In the setting of Lemma~\ref{lem:beta}, the following holds with probability exceeding $1-{\rm e}^{-c_6m}$: 
	\begin{equation}\label{eq:low0}
	\|\bm{B}\bm{x}\|^2\ge 	0.99|\mathcal{S}|\big[1+\log (\nicefrac{m}{|\mathcal{S}|})\big]\ge 0.99\times 1.14^\gamma|\mathcal{S}|\big[1+\log (\nicefrac{m}{|\mathcal{S}|})\big]
	\end{equation}
	provided that $m\ge c_0|\mathcal{S}|\ge c_1n$ for some absolute constants $c_0,\,c_1,\,c_6>0$. 
\end{lemma}


Taking together the upper bound in \eqref{eq:up0}
and the lower bound in \eqref{eq:low0}, one arrives at
\begin{align}\label{eq:fbound}
\frac{\|\bm{B}\bm{u}\|^2}{\|\bm{B}\bm{x}\|^2}\le \frac{C}{1+\log(\nicefrac{m}{|{\mathcal{S}}|})} \buildrel\triangle\over = 
\kappa
\end{align}
where the constant $C:=1.02\times  1.14^{-\gamma}\sqrt{\nicefrac{2^\gamma}{\pi}}\Gamma(\nicefrac{\gamma+1}{2})$ and
which holds with probability at least $1-2{\rm e}^{-c_5n}-{\rm e}^{-c_6m}$, with the proviso that $m\ge c_0|\mathcal{S}|\ge c_1n$. Since $m=\mathcal(O)(n)$, one can then rewrite the probability as $1-c_{3}{\rm e}^{-c_2m}$ for certain constants $c_2,\,c_3>0$.    
To have a sense of the size of $C$, taking our default value $\gamma=0.5$ for instance gives rise to $C=0.7854$.


It is clear that the bound $\kappa$ in \eqref{eq:fbound} can be rendered arbitrarily small  
by taking sufficiently large $\nicefrac{m}{|\mathcal{S}|}$ values (while maintaining $\nicefrac{|\mathcal{S}|}{n}$ to be some constant based on Lemma \ref{lem:low}). With no loss of generality, let us work with $\kappa:=0.001$ in the following.

The wanted upper bound on the distance between the initialization $\bm{z}^0$ and the truth $\bm{x}$ can be obtained based upon similar arguments found in \cite[Section 7.8]{wf}, which are delineated as follows. For unit vectors $\bm{x}$ and $\bm{z}^0$, 
recall from \eqref{eq:sin} that
\begin{equation}
|\bm{x}^\ast\bm{z}^0|^2=\cos^2\theta=1-\sin^2\theta\ge 1-\kappa,
\end{equation}
where $0\le \theta\le \nicefrac{\pi}{2} $ denotes the angle between the spaces spanned by $\bm{z}^0$ and $\bm{x}$, therefore
\begin{align}\label{eq:inequality1}
{\rm dist}^2(\bm{z}^0,\,\bm{x})&\le \|\bm{z}^0\|^2+\|\bm{x}\|^2-2|\bm{x}^\ast\bm{z}^0|\nonumber\\
&\le\left( 2-2\sqrt{1-\kappa}\right)\left\|\bm{x}\right\|^2\nonumber\\
&\approx \kappa\left\|\bm{x}\right\|^2.
\end{align}

As discussed prior to Lemma~\ref{lem:beta}, 
the exact norm $\|\bm{x}\|=1$ is generally not known, and one often scales the unit directional vector found in \eqref{eq:maxeig} by the estimate $\sqrt{\sum_{i=1}^m \nicefrac{\psi_i^2}{m}}$. 
Next the approximation error between the estimated norm $\|\bm{z}^0\|=\sqrt{\sum_{i=1}^m \nicefrac{\psi_i^2}{m}}$ and the true norm $\|\bm{x}\|=1$ is accounted for. Recall from \eqref{eq:maxeig} that the direction of $\bm{x}$ is estimated to be $\tilde{\bm{z}}^0$ (of unity norm). 
Using similar results in \cite[Lemma 7.8 and Section 7.8]{wf}, the following holds with high probability as long as the ratio $m/n$ exceeds some numerical constant 
\begin{equation}\label{eq:inequality2}
\|\bm{z}^0-\tilde{\bm{z}}^0\|=|\|\bm{z}^0\|-1|\le (\nicefrac{1}{20})\|\bm{x}\|.
\end{equation}
Taking the inequalities in \eqref{eq:inequality1}
and \eqref{eq:inequality2} together, it is safe to conclude that
\begin{align}
{\rm dist}(\bm{z}^0,\bm{x})&\le \|\bm{z}^0-\tilde{\bm{z}}^0\|+{\rm dist}(\tilde{\bm{z}}^0,\bm{x})\le (\nicefrac{1}{10})\|\bm{x}\|
\end{align}
which confirms that the initial estimate obeys the relative error $\nicefrac{{\rm dist}(\bm{z}^0,\,\bm{x})}{\|\bm{x}\|}\le 1/10$ for any $\bm{x}\in\mathbb{R}^n$ with  probability $1-c_3{\rm e}^{-c_2m}$, provided that $m\ge c_0|\mathcal{S}|\ge c_1n$ for some numerical constants $c_0,\,c_1,\,c_2,\,c_3>0$.

\subsection{Exact Phase Retrieval from Noiseless Data}\label{subsec:er}

It has been demonstrated that the initial estimate $\bm{z}^0$ obtained by means of the weighted maximal correlation initialization strategy has at most a constant relative error to the globally optimal solution $\bm{x}$, i.e., ${\rm dist}(\bm{z}^0,\bm{x})\le (\nicefrac{1}{10})\|\bm{x}\|$. We demonstrate in the following that starting from such an initial estimate, the RAF iterates (in Step \ref{step:4} of Algorithm~\ref{alg:raf}) converge at a linear rate to the global optimum $\bm{x}$; that is, ${\rm dist}(\bm{z}^t,\bm{x})\le (\nicefrac{1}{10})c^t\|\bm{x}\|$ for some constant $0<c<1$ depending on the step size $\mu>0$, the weighting parameter $
\beta$, and the data $\{(\bm{a}_i;\psi_i)\}_{1\le i\le m}$. This constitutes the second part of the proof of Theorem~\ref{thm:initial}. 
Toward this end, it suffices to show that the iterative updates of RAF is locally contractive within a relatively small neighboring region of the truth $\bm{x}$. Instead of directly coping with the moments in the weights, we establish a conservative result based directly on \cite{taf} and \cite{reshaped}.
Recall first that our gradient flow uses the reweighted gradient 
\begin{equation}
	\nabla\ell_{\rm rw}(\bm{z}:=\frac{1}{m}\sum_{i=1}^m w_i\left(\bm{a}_i^\ast\bm{z}-\psi_i\frac{\bm{a}_i^\ast\bm{z}}{|\bm{a}_i^\ast\bm{z}|}\right)\bm{a}_i
=\frac{1}{m}\sum_{i=1}^m w_i\left(\bm{a}_i^\ast\bm{z}-|\bm{a}_i^\ast\bm{x}|\frac{\bm{a}_i^\ast\bm{z}}{|\bm{a}_i^\ast\bm{z}|}\right)\bm{a}_i\label{eq:rw}
\end{equation}
with weights
\begin{equation}
\label{eq:weights}
w_i=\frac{1}{1+\nicefrac{\beta}{(\nicefrac{|\bm{a}_i^\ast\bm{z}|}{|\bm{a}_i^\ast\bm{x}|})}},\quad 1\le i\le m
\end{equation}
in which the dependence on the iterate index $t$ is ignored for notational brevity.

\begin{proposition}[Local error contraction] \label{prop:lec}
	For arbitrary $\bm{x}\in\mathbb{R}^n$, consider $m$ noise-free measurements $\psi_i=\left|\bm{a}_i^\ast\bm{x}\right|$, $1\le i\le m$. There exist some numerical constants $c_1,\,c_2,\,c_3>0$, and $0<\nu<1$ such that the following holds with probability exceeding $1-c_3{\rm e}^{-c_2m}$
	\begin{equation}\label{eq:contract}
	{\rm dist}^2\!\left(\bm{z}-\mu\nabla \ell_{\rm rw}(\bm{z}),\,\bm{x}\right)\le (1-\nu){\rm dist}^2\!\left(\bm{z},\,\bm{x}\right)
	\end{equation}
	for all $\bm{x},\,\bm{z}\in\mathbb{R}^n$ obeying ${\rm dist}(\bm{z},\,\bm{x})\le (\nicefrac{1}{10})\|\bm{x}\|$, provided that $m\ge c_1 n$ and 
	that the constant step size $\mu\le \mu_0$, where the numerical constant $\mu_0$ depends on the parameter $\beta>0$ and data $\{(\bm{a}_i;\psi_i)\}_{1\le i\le m}$. 
\end{proposition}

Proposition~\ref{prop:lec} suggests that the distance of RAF's successive iterates to the global optimum $\bm{x}$ decreases monotonically once the algorithm's iterate $\bm{z}^t$ enters a small neighboring region around the truth $\bm{x}$. This small-size neighborhood is commonly known as the \emph{basin of attraction}, and has been widely discussed in recent nonconvex optimization contributions; see e.g., \cite{twf}, \cite{reshaped}, \cite{taf}. Expressed differently, RAF's iterates will stay within the region and will be attracted towards $\bm{x}$ exponentially fast as soon as it lands within the basin of attraction. To substantiate Proposition~\ref{prop:lec}, recall the useful analytical tool of the local regularity condition~\cite{wf}, which plays a key role in establishing linear convergence of iterative procedures to the global optimum in \cite{wf}, \cite{twf}, \cite{reshaped}, \cite{taf}, \cite{staf}, \cite{gu2017}, \cite{mtwf}.

For RAF, the reweighted gradient $\nabla\ell_{\rm rw}(\bm{z})$ in  \eqref{eq:rw} 
is said to obey the local regularity condition, or ${\rm LRC}(\mu,\lambda,\epsilon)$ for some constant $\lambda>0$, if the following inequality
\begin{align}
\label{eq:lrc}
\left\langle\nabla\ell_{\rm rw}(\bm{z}),\,\bm{h}\right\rangle \ge \frac{\mu}{2}\left\|\nabla\ell_{\rm rw}(\bm{z})\right\|^2+\frac{\lambda}{2}\left\|\bm{h}\right\|^2
\end{align}
holds for all $\bm{z}\in\mathbb{R}^n$ such that $\left\|\bm{h}\right\|=\left\|\bm{z}-\bm{x}\right\|
\le \epsilon\left\|\bm{x}\right\|$ for some constant $0<\epsilon<1$, where the ball given by $\left\|\bm{z}-\bm{x}\right\|\le \epsilon\left\|\bm{x}\right\|$ is the so-termed \emph{basin of attraction}.

Realizing $\bm{h}:=\bm{z}-\bm{x}$, algebraic manipulations in conjunction with the regularity property \eqref{eq:lrc} confirms
 \begin{align}
{\rm dist}^2\!\left(\bm{z}-{\mu}\nabla\ell_{\rm rw}(\bm{z}),\,\bm{x}\right)&=\left\|\bm{z}-\mu\nabla\ell_{\rm rw}(\bm{z})-\bm{x}
\right\|^2\nonumber\\
&=\left\|\bm{h}\right\|^2-2\mu\left\langle\bm{h},\nabla
\ell_{\rm rw}(\bm{z})\right\rangle+\left\|\mu\nabla\ell_{\rm rw}(\bm{z})\right\|^2\label{eq:lastterm}\\
&\le \left\|\bm{h}\right\|^2-2\mu\left( \frac{\mu}{2}\left\|\nabla\ell_{\rm rw}(\bm{z})\right\|^2+\frac{\lambda}{2}\left\|\bm{h}\right\|^2\right)+\left\|\mu\nabla\ell_{\rm rw}(\bm{z})\right\|^2\nonumber\\
&=\left(1-\lambda\mu\right)\left\|\bm{h}\right\|^2=\left(1-\lambda\mu\right){\rm dist}^2(\bm{z},\,\bm{x})\label{eq:lrc1}
\end{align}
for all points $\bm{z}$ adhering to $\left\|\bm{h}\right\|\le \epsilon\left\|\bm{x}\right\|$. 
It is self-evident that if the regularity condition ${\rm LRC}(\mu,\lambda,\epsilon)$ can be established for RAF, our ultimate goal of proving the local error contraction in \eqref{eq:contract} follows straightforwardly upon setting $\nu:=\lambda\mu$.

\subsubsection{Proof of the local regularity condition in~\eqref{eq:lrc}}

The first step of proving the local regularity condition in~\eqref{eq:lrc} is to control the size of the reweighted gradient $\nabla\ell_{\rm rw}(\bm{z})$, i.e., to upper bound the last term in~\eqref{eq:lastterm}. To start, let rewrite the reweighted gradient in a compact matrix-vector representation
\begin{equation}
\nabla\ell_{\rm rw}(\bm{z})=\frac{1}{m}\sum_{i=1}^mw_i\left(\bm{a}_i^\ast\bm{z}-|\bm{a}_i^\ast\bm{x}|\frac{\bm{a}_i^\ast\bm{z}}{\left|\bm{a}_i^\ast\bm{z}\right|}\right)\bm{a}_i\buildrel\triangle\over =\frac{1}{m}{\rm diag}(\bm{w})\bm{A}\bm{v}\label{eq:eqav}
\end{equation}
where ${\rm diag}(\bm{w})\in\mathbb{R}^{n\times n}$ is a diagonal matrix holding in order entries of $\bm{w}:=[w_1~\cdots~w_m]^\ast\in\mathbb{R}^m$ on its main diagonal, and 
$\bm{v}:=[v_1~\cdots~v_m]^\ast\in\mathbb{R}^m$ with $v_i:=\bm{a}_i^\ast\bm{z}-|\bm{a}_i^\ast\bm{x}|\frac{\bm{a}_i^\ast\bm{z}}{|\bm{a}_i^\ast\bm{z}|}$.   
Based on the definition of the induced matrix $2$-norm (or the matrix spectral norm), it is easy to check that
\begin{align}
\|\nabla\ell_{\rm rw}(\bm{z})\|&=\left\|\frac{1}{m}{\rm diag}(\bm{w})\bm{A}\bm{v}\right\|\nonumber\\&\le \frac{1}{m}\|{\rm diag}(\bm{w})\|\cdot\|\bm{A}\|\cdot\|\bm{v}\|\nonumber\\
&\le \frac{1+\delta'}{\sqrt{m}}\|\bm{v}\|
\label{eq:ineqvi}
\end{align}
where we have used the inequalities $\|{\rm diag}(\bm{w})\|\le 1$ due to $w_i\le 1$ for all $1\le i\le m$, and $\|\bm{A}\|\le (1+\delta')\sqrt{m}$ for some constant $\delta'>0$ according to \cite[Theorem 5.32]{chap2010vershynin}, provided that $m/n$ is sufficiently large.

The task therefore remains to bound $\|\bm{v}\|$ in \eqref{eq:ineqvi}, which is addressed next. To this end, notice that 
\begin{align}\label{eq:vbound}
\|\bm{v}\|^2=\sum_{i=1}^m\!\left(\bm{a}_i^\ast\bm{z}-|\bm{a}_i^\ast\bm{x}|\frac{\bm{a}_i^\ast\bm{z}}{|\bm{a}_i^\ast\bm{z}|}\right)^2&\le\sum_{i=1}^m\left(|\bm{a}_i^\ast\bm{z}|-|\bm{a}_i^\ast\bm{x}|\right)^2\nonumber\\
&\le \sum_{i=1}^m\left(\bm{a}_i^\ast\bm{z}-\bm{a}_i^\ast\bm{x}\right)^2\nonumber\\
&=\sum_{i=1}^m(\bm{a}_i^\ast\bm{h})^2\le (1+\delta'')^2m\|\bm{h}\|^2
\end{align}
for some numerical constant $\delta''>0$,
where the last can be obtained using \cite[Lemma 3.1]{phaselift} and which holds with probability at least $1-{\rm e}^{-c_2m}$ as long as $m>c_1n$ holds true.

Combing the results in \eqref{eq:ineqvi} and \eqref{eq:vbound} and taking $\delta>0$ larger than the constant $(1+\delta')(1+\delta'')-1$, the size of $\nabla\ell_{\rm rw}(\bm{z})$ can be bounded as follows
\begin{equation}
\label{eq:gradbound}
\|\nabla\ell_{\rm rw}(\bm{z})\|\le (1+\delta)\|\bm{h}\|
\end{equation}  
which holds with probability $1-{\rm e}^{-c_2m}$,
with a proviso that $m/n$ exceeds some numerical constant $c_7>0$.
This result indeed asserts that the reweighted gradient of the objective function $L(\bm{z})$
or the search direction employed in our RAF algorithm is well behaved, implying that the function value along the iterates does not change too much.

In order to prove the LRC, it suffices to show that the reweighted gradient $\nabla\ell_{\rm rw}(\bm{z})$ ensures sufficient descent, 
that is, there exists a numerical constant $c>0$ such that along the search direction $\nabla \ell_{\rm rw}(\bm{z})$ the following uniform lower bound holds
\begin{equation}
\left\langle\nabla
\ell_{\rm rw}(\bm{z}),\,\bm{h}\right\rangle \ge c \|\bm{h}\|^2
\end{equation} 
which will be addressed next. Formally, this can be summarized in the following proposition, whose proof is deferred to Appendix \ref{prop:right}.

\begin{proposition}\label{prop:right}
	Fixing any sufficiently small constant $\epsilon>0$, consider the noise-free measurements $\psi_i=|\bm{a}_i^\ast\bm{x}|$, $1\le i\le m$. There exist some numerical constants $c_1,\,c_2,\,c_3>0$ such that the following holds with probability at least $1-c_3{\rm e}^{-c_2m}$:
	\begin{equation}\label{eq:innerproduct}
	\left\langle\bm{h},\nabla	\ell_{\rm rw}(\bm{z})\right\rangle \ge \left[\frac{1-\zeta_1-\epsilon}{1+\beta(1+\eta)}-2(\zeta_2+\epsilon )-\frac{2(0.1271-\zeta_2+\epsilon)}{1+\nicefrac{\beta}{k}}
	\right]\|\bm{h}\|^2
	\end{equation} 
	for all $\bm{x},\,\bm{z}\in\mathbb{R}^n$ obeying $\|\bm{h}\|\le \frac{1}{10}\|\bm{x}\|$, provided that $m/n>c_1$, and that $\beta\ge 0$ is small enough.
\end{proposition}

Taking the results in  \eqref{eq:innerproduct} and \eqref{eq:gradbound} together back to \eqref{eq:lrc}, one concludes that the local regularity condition holds for $\mu$ and $\lambda$ obeying the following
\begin{align}
\frac{1-\zeta_1-\epsilon}{1+\beta(1+\eta)}-2(\zeta_2+\epsilon )-\frac{2(0.1271-\zeta_2+\epsilon)}{1+\nicefrac{\beta}{k}}
\ge \frac{\mu}{2}(1+\delta)^2+\frac{\lambda}{2}.
\end{align}
For instance, take $\beta=2$, 
$k=5$, $\eta=0.5$, and $\epsilon=0.001$, we have $\zeta_1=0.8897$ and $\zeta_2=0.0213$, thus confirming that 
$
\langle\ell_{\rm rw}(\bm{z}),\bm{h}\rangle\ge 0.1065\|\bm{h}\|^2.
$ Setting further $\delta=0.001$ leads to
\begin{equation}
0.1065\ge 0.501\mu+0.5\lambda
\end{equation}
which concludes the proof of the local regularity condition in \eqref{eq:lrc}. The local error contraction in \eqref{eq:contract} follows directly from substituting the local regularity condition into \eqref{eq:lrc1}, hence validating Proposition \ref{prop:lec}.


\section{Conclusions}\label{sec:con}
This paper put forth a linear-time algorithm 
termed reweighted amplitude flow (RAF) 
for solving systems of random quadratic equations. Our novel procedure effects two consecutive stages, namely, a weighted maximal correlation initialization that is attainable based upon a few power or Lanczos iterations, and a sequence of simple iteratively reweighted generalized gradient iterations for the nonconvex nonsmooth least-squares loss function. Our RAF approach is conceptually simple, easy-to-implement, as well as numerically scalable and effective. It was also demonstrated to achieve the optimal sample and computational complexity orders. 
Substantial numerical tests
using both synthetic data and real images
corroborated the superior performance of RAF over state-of-the-art iterative solvers. Empirically, RAF solves a set of random quadratic equations in the high-dimensional regime with large probability so long as a unique solution exists, where the number $m$ of equations in the real-valued Gaussian case can be as small as $2n-1$ with $n$ being the number of unknowns to be recovered. 

Future research extensions include studying robust and/or sparse phase retrieval and (semi-definite) matrix recovery by means of (stochastic) reweighted amplitude flow counterparts \cite{sparta,duchi2017,mtwf,tit2015ccg,acha2015krt}. Exploiting the possibility of leveraging suitable (re)weighting regularization to improve empirical performance of other nonconvex iterative procedures such as \cite{duchi2017}, \cite{mtwf}, \cite{gu2017}
is worth investigating as well.


\subsection*{Acknowledgments}

The authors would like to thank John C. Duchi for his helpful feedback on our initialization.


\appendix

\section{Proof details}
By homogeneity of \eqref{eq:quad}, it suffices to work with the case where $\|\bm{x}\|=1$. 

\subsection{Proof of Lemma \ref{lem:beta}}
\label{sec:lembeta}
It is easy to check that 
\begin{align}
\frac{1}{2}\big\|\bm{x}\bm{x}^\ast-{\bm{z}}^0(\bm{z}^0)^\ast\big\|_F
^2&=\frac{1}{2}\|\bm{x}\|^4+\frac{1}{2}\|{\bm{z}}^0\|^4-|\bm{x}^\ast{\bm{z}}^0|^2\nonumber\\
&=1-|\bm{x}^\ast{\bm{z}}^0|^2\nonumber\\
&=1-\cos^2\theta\label{eq:ineq1}
\end{align}
where $0\le \theta\le \nicefrac{\pi}{2}$ denotes the angle between the hyperplanes spanned by $\bm{x}$ and $\bm{z}^0$. Letting $(\bm{z}^0)^\perp\in \mathbb{R}^n$ be a unit vector orthogonal to ${\bm{z}}^0$ and have a nonnegative inner-product with $\bm{x}$, then $\bm{x}$ can be uniquely expressed as a linear communication of $\bm{z}^0$ and $(\bm{z}^0)^\perp$, yielding 
\begin{equation}
\label{eq:orth1}
\bm{x}={\bm{z}}^0\cos\theta+({\bm{z}}^0)^{\perp}\sin\theta.
\end{equation}
Likewise, introduce the unit vector $\bm{x}^{\perp}$ to be orthogonal to $\bm{x}$ and to have a nonnegative inner-product with $(\bm{z}^0)^\perp$. Therefore, $\bm{x}^\perp$ can be uniquely written as 
\begin{equation}
\label{eq:orth2}
\bm{x}^{\perp}:=-\bm{z}^0\sin\theta+(\bm{z}^0)^{\perp}\cos\theta.
\end{equation}

Recall from \eqref{eq:maxeig} (after ignoring the normalization factor $\nicefrac{1}{|\mathcal{S}|}$) that $\bm{z}^0$ is the solution to the principal component analysis (PCA) problem
\begin{align}\label{eq:maxeig1}
{\bm{z}}^0:=\arg\max_{\|\bm{z}\|=1}~&~\bm{z}^\ast\bm{B}^\ast\bm{B}\bm{z}.
\end{align}		
Therefore, it holds that
$\bm{B}^\ast\bm{B}\bm{z}^0
=\lambda_{1}{\bm{z}}^0$, where $\lambda_1>0$ is the largest eigenvalue of $\bm{B}^\ast\bm{B}$. 
Multiplying \eqref{eq:orth1} and \eqref{eq:orth2} by $\bm{B}$ from the left gives rise to  
\begin{subequations}\label{eq:prem}
	\begin{align}
	\bm{B}	\bm{x}&=	\bm{B}\bm{z}^0\cos\theta+	\bm{B}(\bm{z}^0)^{\perp}\sin\theta\label{eq:prem1},\\
	\bm{B}	\bm{x}^{\perp}&=-	\bm{B}\bm{z}^0\sin\theta+	\bm{B}(\bm{z}^0)^{\perp}\cos\theta\label{eq:prem2}.
	\end{align}	
\end{subequations}
Taking the $2$-norm square of both sides in \eqref{eq:prem1} and \eqref{eq:prem2}
yields  
\begin{subequations}\label{eq:prem}
	\begin{align}
	\|\bm{B}\bm{x}\|^2&=\|\bm{B}\bm{z}^0\|^2\cos^2\theta+\|\bm{B}(\bm{z}^0)^{\perp}\|^2\sin^2\theta\label{eq:prem11},\\
	\|\bm{B}\bm{x}^\perp\|^2&=\|\bm{B}\bm{z}^0\|^2\sin^2\theta+\|\bm{B}(\bm{z}^0)^{\perp}\|^2\cos^2\theta\label{eq:prem21},
	\end{align}	
\end{subequations}
where the cross-terms disappear due to $(\bm{z}^0)^\ast\bm{B}^\ast\bm{B}(\bm{z}^0)^\perp=\lambda_1(\bm{z}^0)^\ast(\bm{z}^0)^\perp=0$ according to the definition of $(\bm{z}^0)^\perp$.

With the relationships established in \eqref{eq:prem},
construct now the following
\begin{align*}
&\|\bm{B}\bm{x}\|^2\sin^2\theta-\|\bm{B}\bm{x}^\perp\|^2\\
&=(\|\bm{B}\bm{z}^0\|^2\cos^2\theta+\|\bm{B}(\bm{z}^0)^{\perp}\|^2\sin^2\theta)\sin^2\theta
-(\|\bm{B}\bm{z}^0\|^2\sin^2\theta+\|\bm{B}(\bm{z}^0)^{\perp}\|^2\cos^2\theta)\\
&=\big(
\|\bm{B}\bm{z}^0\|^2\cos^2\theta-\|\bm{B}\bm{z}^0\|^2+\|\bm{B}(\bm{z}^0)^\perp\|^2\sin^2\theta \big)\sin^2\theta-\|\bm{B}(\bm{z}^0)^\perp \|^2\cos^2\theta\\
&=\big(\|\bm{B}(\bm{z}^0)^\perp\|^2-\|\bm{B}\bm{z}^0\|^2\big)\sin^4\theta-\|\bm{B}(\bm{z}^0)^\perp\|^2\cos^2\theta\\
&\le 0
\end{align*}
where $\bm{B}^\ast\bm{B}\succeq \bm{0}$, so $\|\bm{B}(\bm{z}^0)^\perp\|^2-\|\bm{B}\bm{z}^0\|^2\le 0$ holds for any unit vector $(\bm{z}^0)^\perp\in\mathbb{R}^n$ because $\bm{z}^0$ maximizes the term in \eqref{eq:maxeig},  
hence yielding
\begin{equation}\label{eq:sin}
\sin^2\theta=1-\cos^2\theta \le \frac{\|\bm{B}\bm{x}^\perp\|^2}{\|\bm{B}\bm{x}\|^2}.
\end{equation}
Plugging \eqref{eq:ineq1} into above, \eqref{eq:mse} follows directly from setting
$\bm{u}=\bm{x}^\perp$.

\subsection{Proof of Lemma \ref{lem:up}}\label{sec:proofup}

Let $\{\bm{b}_i^\ast\}_{1\le i\le |\mathcal{S}|}$ denote rows of $\bm{B}\in\mathbb{R}^{|\mathcal{S}|\times n}$, which are obtained by scaling rows of  $\bm{A}_{\mathcal{S}}:=\{\bm{a}_i^\ast\}_{i\in\mathcal{S}}\in\mathbb{R}^{|\mathcal{S}|\times n}$ by weights $\{w_i=\psi_i^{\nicefrac{\gamma}{2}}\}_{i\in\mathcal{S}}$ [cf. \eqref{eq:eig}]. Since $\bm{x}=\bm{e}_1$, then $\bm{\psi}=|\bm{A}\bm{e}_1|=|\bm{A}_{ 1}|$, namely, the index set $\mathcal{S}$ depends solely on the first column of $\bm{A}$, and is independent of the other columns of $\bm{A}$. 
In this direction, 
partition accordingly $\bm{A}^{\mathcal{S}}:=[\bm{A}^{\mathcal{S}}_1~\bm{A}^{\mathcal{S}}_r]$, where $\bm{A}^{\mathcal{S}}_1\in\mathbb{R}^{|\mathcal{S}|\times 1}$ denotes the first column of $\bm{A}^{\mathcal{S}}$, and $\bm{A}^{\mathcal{S}}_r\in\mathbb{R}^{|\mathcal{S}|\times (n-1)}$ collects the remaining ones. Likewise, partition $\bm{B}=[\bm{B}_1~\bm{B}_r]$ with $\bm{B}_1\in\mathbb{R}^{|\mathcal{S}|\times 1}$ and $\bm{B}_r\in\mathbb{R}^{|\mathcal{S}|\times (n-1)}$. By the argument above, rows of $\bm{A}^\mathcal{S}$ are mutually independent, and they follow i.i.d. Gaussian distribution with mean $\bm{0}$ and covariance matrix $\bm{I}_{n-1}$. Furthermore, the weights $\psi_{i}^{\nicefrac{\gamma}{2}}=|\bm{a}_{i}^\ast\bm{e}_1|^{\nicefrac{\gamma}{2}}=|a_{i,1}|^{\nicefrac{\gamma}{2}}$, $\forall i\in\mathcal{S}$ are also independent of the entries in $\bm{A}^\mathcal{S}$.
As a consequence, rows of $
\bm{B}_r$ are mutually independent of each other, and one can explicitly write its $i$-th row as $\bm{b}_{r,i}=|\bm{a}_{[i]}^\ast\bm{e}_1|^{\nicefrac{\gamma}{2}}\bm{a}_{[i],\backslash 1}=|a_{[i],1}|^{\nicefrac{\gamma}{2}}
\bm{a}_{[i],\backslash 1}$, where $\bm{a}_{[i],\backslash 1}\in\mathbb{R}^{n-1}$ is obtained through removing the first entry of $\bm{a}_{[i]}$. It is easy to verify that $\mathbb{E}[\bm{b}_{r,i}]=\bm{0}$, and $\mathbb{E}[\bm{b}_{r,i}\bm{b}_{r,i}^\ast]=C_\gamma\bm{I}_{n-1}$, where the constant $C_\gamma:=\sqrt{\nicefrac{2^\gamma}{\pi}} \Gamma(\nicefrac{\gamma+1}{2}) \|\bm{x}\|^{\gamma}=\sqrt{\nicefrac{2^\gamma}{\pi}} \Gamma(\nicefrac{\gamma+1}{2}) $, and $\Gamma(\cdot)$ is the Gamma function. 

Given
$\bm{x}^\ast\bm{x}^\perp=\bm{e}_1^\ast\bm{x}^\perp=0$, one can write $\bm{x}^\perp=[0~\bm{r}^\ast]^\ast$ with any unit vector $\bm{r}\in\mathbb{R}^{n-1}$, hence
\begin{equation}\label{eq:su}
\|\bm{B}\bm{x}^\perp\|^2
=\|\bm{B}[0~\bm{r}^\ast]^\ast\|^2=\|\bm{B}_r\bm{r}\|^2
\end{equation}
with independent subgaussian rows $\bm{b}_{r,i}=|a_{j,1}|^{\nicefrac{\gamma}{2}}\bm{a}_{j,\backslash 1}$ if $0\le \gamma\le 1$.
Standard concentration results on the sum of random positive semi-definite matrices composed of independent non-isotropic subgaussian rows~\cite[Remark 5.40.1]{chap2010vershynin} assert that 
\begin{equation}\label{eq:iso}
\left\|\frac{1}{|\mathcal{S}|}\bm{B}_r^\ast\bm{B}_r-C_\gamma\bm{I}_{n-1} \right\|\le \delta
\end{equation}
holds with probability at least $1-2{\rm e}^{-c_5 n}$ provided that $\nicefrac{|\mathcal{S}|}{n}$ is larger than some positive constant. Here,
$\delta>0$ is a numerical constant that can take arbitrarily small values, and $c_5>0$ is a constant depending on $\delta$. 
With no loss of generality, take $\delta:=0.01C_\gamma$ in~\eqref{eq:iso}. For any unit vector $\bm{r}\in\mathbb{R}^{n-1}$, the following holds with probability at least $1-2{\rm e}^{-c_5 n}$ 
\begin{equation}
\left\|\frac{1}{|\mathcal{S}|}\bm{r}^\ast\bm{B}_r^\ast\bm{B}_r\bm{r}-C_\gamma\bm{r}^\ast\bm{r} \right\|\le \delta \bm{r}^\ast\bm{r}=\delta
\end{equation}
or
\begin{equation}
\left\|\bm{B}_r\bm{r}\right\|^2=\bm{r}^\ast\bm{B}_r^\ast\bm{B}_r\bm{r}\le 1.01C_\gamma|\mathcal{S}|.
\end{equation}
Taking the last back to \eqref{eq:su} confirms that
\begin{equation}\label{eq:up}
\|\bm{B}\bm{x}^\perp\|^2\le1.01C_\gamma|\mathcal{S}|
\end{equation}
holds with probability at least $1-2{\rm e}^{-c_5 n}$ if $\nicefrac{|\mathcal{S}|}{n}$ exceeds some constant. Note that $c_5$ depends on the maximum subgaussian norm of the rows  $\bm{b}_i$ in $\bm{B}_r$, and we assume without loss of generality $c_5\ge 1/2$. Therefore, one confirms that the numerator $\|\bm{B}\bm{u}\|^2$ in~\eqref{eq:mse}
is upper bounded via replacing $\bm{x}^\perp$ with $\bm{u}$ in~\eqref{eq:up}.

\subsection{Proof of Lemma \ref{lem:low}}\label{sec:prooflow}

This section is devoted to obtaining a meaningful lower bound for the denominator $\|\bm{B}\bm{x}\|^2$ in~\eqref{eq:low0}. Note first that 
$$\|\bm{B}\bm{x}\|^2=\sum_{i=1}^{|\mathcal{S}|}\|\bm{b}_i^\ast\bm{x}\|^{2}=\sum_{i=1}^{|\mathcal{S}|}\psi_{[i]}^\gamma|\bm{a}_{[i]}^\ast\bm{x}|^2=\sum_{i=1}^{|\mathcal{S}|}|\bm{a}_{[i]}^\ast\bm{x}|^{2+\gamma}.
$$
Taking without loss of generality $\bm{x}=\bm{e}_1$, the term on the right side of the last equality reduces to 
\begin{equation}
\label{eq:lowerb}
\|\bm{B}\bm{x}\|^2=\sum_{i=1}^{|\mathcal{S}|}|a_{[i],1}|^{2+\gamma}.
\end{equation}
Since $a_{[i],1}$ follows the standard normal distribution, the probability density function (pdf) of random variables $|a_{[i],1}|^{2+\gamma}$ can be given in closed form as 
\begin{equation}\label{eq:pdf1}
p(t)=\sqrt{\frac{2}{\pi}}\cdot\frac{1}{2+\gamma} t^{-\frac{1+\gamma}{2+\gamma}}{\rm e}^{-\frac{1}{2}t^{\frac{2}{2+\gamma}}},\quad t> 0
\end{equation}
which is rather complicated and whose cumulative density function (cdf) 
does not come in closed-form in general. Therefore, instead of dealing with the pdf in \eqref{eq:pdf1} directly, we shall take a different route by deriving a lower bound that is a bit looser yet suffices for our purpose, which is detailed as follows.  

Since $|a_{[|\mathcal{S}|],1}|\le \cdots\le |a_{[2],1}|\le |a_{[1],1}|$, then it holds for all $1\le i\le |\mathcal{S}|$ that $|a_{[i],1}|^{2+\gamma}\ge |a_{[|\mathcal{S}|],1}|^\gamma a_{[i],1}^2$, therefore yielding
\begin{equation}\label{eq:bbb}
\|\bm{B}\bm{x}\|^2=\sum_{i=1}^{|\mathcal{S}|}|a_{[i],1}|^{2+\gamma}\ge |a_{[|\mathcal{S}|],1}|^\gamma\sum_{i=1}^{|\mathcal{S}|}a_{[i],1}^{2}.
\end{equation}
Hence, we next demonstrate that deriving a lower bound for $\|\bm{B}\bm{x}\|^2$ suffices to derive a lower bound for the summation on the right hand side above. The latter can be achieved by appealing to a result in \cite[Lemma 3]{taf}, which for completeness is included in the following.
\begin{lemma}\label{lem:lowtaf}
	For arbitrary unit vector $\bm{x}\in\mathbb{R}^n$, 
	let $\psi_i=|\bm{a}_i^\ast\bm{x}|$, $1\le i\le m$ be $m$ noiseless measurements.
	Then with probability at least $1-{\rm e}^{-c_2m}$, the following holds:
	\begin{equation}\label{eq:low0taf}
	\sum_{i=1}^{|\mathcal{S}|} a_{[i],1}^2
	\ge {0.99|\mathcal{S}|}\big[1+\log (\nicefrac{m}{|\mathcal{S}|})\big]
	\end{equation}
	provided that $m\ge c_0|\mathcal{S}|\ge c_1 n$ for some numerical constants $c_0,\,c_1,\,c_2>0$. 
\end{lemma}

Combining the results in Lemma \ref{lem:lowtaf} and \eqref{eq:bbb} together, one further establishes that
\begin{equation}\label{eq:bxbound}
\|\bm{B}\bm{x}\|^2\ge |a_{[|\mathcal{S}|],1}|^\gamma \sum_{i=1}^{|\mathcal{S}|}a_{[i],1}^2\ge |a_{[|\mathcal{S}|],1}|^\gamma 
\cdot	0.99|\mathcal{S}|\big[1+\log (\nicefrac{m}{|\mathcal{S}|})\big].
\end{equation}
The task remains to estimate the size of $|a_{[|\mathcal{S}|],1}|$, which we recall is the $|\mathcal{S}|$-th largest among the $m$ independent realizations $\{\psi_{i}=|a_{i,1}|\}_{1\le i\le m}$. 
Taking $\gamma=-1$ in \eqref{eq:pdf1} gives the pdf of the half-normal distribution
\begin{equation}
\label{eq:pdf2}
p(t)=\sqrt{\frac{2}{\pi}}{\rm e}^{-\frac{1}{2}t^2},\quad t> 0
\end{equation}
whose corresponding cdf is
\begin{equation}
\label{eq:cdf}
F(\tau)={\rm erf}(\nicefrac{\tau}{\sqrt{2}}).
\end{equation}

Setting $F(\tau_{|\mathcal{S}|}):=1-\nicefrac{|\mathcal{S}|}{m}$ or using the complementary cdf $\nicefrac{|\mathcal{S}|}{m}:={\rm erfc}(\nicefrac{\tau}{\sqrt{2}})$ based on the  complementary error function
gives rise to an estimate of the size of the $|\mathcal{S}|$-th largest [or equivalently, the $(m-|\mathcal{S}|)$-th smallest] entry in the $m$ realizations, namely
\begin{equation}
\tau_{|\mathcal{S}|}=\sqrt{2} \,{\rm erfc}^{-1}(\nicefrac{\mathcal{|S}|}{m})
\end{equation}  
where ${\rm erfc}^{-1}(\cdot)$ represents the inverse complementary error function. In the sequel, we show that the deviation of the $|\mathcal{S}|$-th largest realization $\psi_{|\mathcal{S}|}$ from its expected value $\tau_{|\mathcal{S}|}$ found above is bounded with high probability.

For random variable $\psi=|a|$ with $a$ obeying the standard Gaussian distribution, consider the event $\psi\le \tau_{|\mathcal{S}|}-\delta$ for fixed constant $\delta>0$. Define the indicator random variable $\chi=\mathbb{1}_{\{\psi\le \tau_{|\mathcal{S}|}-\delta\}}$, whose expectation can be obtained by substituting $\tau=\tau_{|\mathcal{S}|}-\delta$ into the pdf in \eqref{eq:cdf}
\begin{equation}
\label{eq:chiexp}
\mathbb{E}[\chi_i]={\rm erf}(\nicefrac{\tau_{|\mathcal{S}|}-\delta}{\sqrt{2}}).
\end{equation}
Consider now the $m$ independent copies $\{\chi_i=\mathbb{1}_{\{\psi_i\le \tau_{|\mathcal{S}|}-\delta 
	\}}\}_{1\le i\le m}$ of $\chi$, and the following holds
\begin{align}
\mathbb{P}(\psi_{|\mathcal{S}|}\le \tau_{|\mathcal{S}|}-\delta)&=\mathbb{P}\Big(\sum_{i=1}^m\chi_i\le m-|\mathcal{S}|\Big)\nonumber\\
&= \mathbb{P}\Big(\frac{1}{m}\sum_{i=1}^m\big(\chi_i-\mathbb{E}[\chi_i]\big)\le 1-\frac{|\mathcal{S}|}{m}-\mathbb{E}[\chi_i]
\Big)\label{eq:chivar}.
\end{align} 

Clearly, random variables $\chi_i$ are bounded, so they are sub-gaussian \cite{chap2010vershynin}. For notational brevity, let $t:=1-\nicefrac{|\mathcal{S}|}{m}-\mathbb{E}[\chi_i]=1-\nicefrac{|\mathcal{S}|}{m}-{\rm erf}(\nicefrac{\tau_{|\mathcal{S}|}-\delta}{\sqrt{2}})$.
Appealing to a large deviation inequality for sums of independent
sub-gaussian random variables, one establishes that
\begin{equation}
\label{eq:subchi}
\mathbb{P}(\psi_{|\mathcal{S}|}\le \tau_{|\mathcal{S}|}-\delta)= \mathbb{P}\Big(\frac{1}{m}\sum_{i=1}^m\big(\chi_i-\mathbb{E}[\chi_i]\big)\le 1-\frac{|\mathcal{S}|}{m}-\mathbb{E}[\chi_i]
\Big)\le {\rm e}^{-c_5mt^2}
\end{equation}
where $c_5>0$ is some absolute constant. 
On the other hand, using the definition of the error function and properties of integration gives rise to
\begin{align}
\label{eq:tbound}
t=1-\nicefrac{|\mathcal{S}|}{m}-{\rm erf}(\nicefrac{\tau_{|\mathcal{S}|}-\delta}{\sqrt{2}})=\frac{2}{\sqrt{\pi}} \int_{\nicefrac{(\tau_{|\mathcal{S}|}-\delta)}{\sqrt{2}}}^{\nicefrac{\tau_{|\mathcal{S}|}}{\sqrt{2}}} {\rm e}^{-s^2}{\rm d}s\ge \sqrt{\frac{2}{\pi}}\delta {\rm e}^{-\frac{\tau_{|\mathcal{S}|}^2}{2}}\ge  \sqrt{\frac{2}{\pi}}\delta.
\end{align}
Taking the results in \eqref{eq:subchi} and \eqref{eq:tbound} together, one concludes that
fixing any constant $\delta>0$, the following holds with probability at least
$
1-{\rm e}^{-c_2m}$: 
$$\psi_{|\mathcal{S}|}\ge \tau_{|\mathcal{S}|}-\delta\ge \sqrt{2} \,{\rm erfc}^{-1}(\nicefrac{\mathcal{|S}|}{m})-\delta$$ 
where the constant $c_2:={\nicefrac{2}{\pi}}\cdot c_5\delta^2$. Furthermore, choosing without loss of generality $\delta:=0.01\tau_{|\mathcal{S}|}$ above leads to $\psi_{|\mathcal{S}|}\ge 1.4\,{\rm erfc}^{-1}(\nicefrac{|\mathcal{S}|}{m})$. 

Substituting the last inequality into \eqref{eq:bxbound} and under our working assumption $\nicefrac{|\mathcal{S}|}{m}\le 0.25$, one readily obtains that
\begin{equation}
\|\bm{B}\bm{x}\|^2\ge [1.4\,
{\rm erfc}^{-1}(\nicefrac{|\mathcal{S}|}{m}) ]^\gamma
\cdot	0.99|\mathcal{S}|\big[1+\log (\nicefrac{m}{|\mathcal{S}|})\big]\ge 0.99\cdot1.14^\gamma|\mathcal{S}|\big[1+\log (\nicefrac{m}{|\mathcal{S}|})\big]
\end{equation}
which holds with probability exceeding $1-{\rm e}^{-c_2m}$ for some absolute constant $c_2>0$, concluding the proof of Lemma \ref{lem:low}.

\subsection{Proof of Proposition \ref{lem:low}}\label{prop:right}

To proceed, let us introduce the following events for all $1\le i\le m$:
\begin{align}
\mathcal{D}_i&:=\big\{(\bm{a}_i^\ast\bm{x})(\bm{a}_i^\ast\bm{z})<0\big\}\label{eq:eset} \\
\mathcal{E}_i&:=	\left\{\frac{|\bm{a}_i^\ast\bm{z}|}{|\bm{a}_i^\ast\bm{x}|}\ge \frac{1}{1+\eta}
\right\}\label{eq:gset}
\end{align}
for some fixed constant $\eta>0$, 
in which the former corresponds to the gradients involving wrongly estimated signs, namely, $\frac{\bm{a}_i^\ast\bm{z}}{|\bm{a}_i^\ast\bm{z}|}\ne \frac{\bm{a}_i^\ast\bm{x}}{|\bm{a}_i^\ast\bm{x}|}$, and the second will be useful for deriving error bounds. Based on the definition of $\mathcal{D}_i$ and with $\mathbb{1}_{\mathcal{D}_i}$ denoting the indicator function of the event $\mathcal{D}_i$, we have
\begin{align}
\langle\ell_{\rm rw}(\bm{z}),\bm{h}\rangle&=\frac{1}{m}\sum_{i=1}^m w_i\!\left(\bm{a}_i^\ast\bm{z}-|\bm{a}_i^\ast\bm{x}|\frac{\bm{a}_i^\ast\bm{z}}{|\bm{a}_i^\ast\bm{z}|}
\right)(\bm{a}_i^\ast\bm{h})\nonumber\\
&=\frac{1}{m}\sum_{i=1}^m w_i\!\left(\bm{a}_i^\ast\bm{h}+\bm{a}_i^\ast\bm{x}-|\bm{a}_i^\ast\bm{x}|\frac{\bm{a}_i^\ast\bm{z}}{|\bm{a}_i^\ast\bm{z}|}
\right)(\bm{a}_i^\ast\bm{h})\nonumber\\
&=\frac{1}{m}\sum_{i=1}^m w_i(\bm{a}_i^\ast\bm{h})^2+\frac{1}{m}\sum_{i=1}^m2w_i
\big(\bm{a}_i^\ast\bm{x}\big)(\bm{a}_i^\ast\bm{h})\mathbb{1}_{\mathcal{D}_i}\nonumber\\
&\ge \frac{1}{m}\sum_{i=1}^mw_i(\bm{a}_i^\ast\bm{h})^2-\frac{1}{m}\sum_{i=1}^m2w_i\big|\bm{a}_i^\ast\bm{x}\big|\big|\bm{a}_i^\ast\bm{h}\big|\mathbb{1}_{\mathcal{D}_i} \label{eq:lrcfinal}.	
\end{align}  

In the following, we will derive a lower bound for the term on the right hand side of \eqref{eq:lrcfinal}. To be specific, a lower bound for the first term $\frac{1}{m}\sum_{i=1}^mw_i(\bm{a}_i^\ast\bm{h})^2$
and an upper bound for the second term $\frac{1}{m}\sum_{i=1}^m2w_i\big|\bm{a}_i^\ast\bm{x}\big|\big|\bm{a}_i^\ast\bm{h}\big|\mathbb{1}_{\mathcal{D}_i} $ will be obtained, which occupies Lemmas \ref{le:1sttermraf} and \ref{le:2ndtermraf}, with their proofs postponed to Appendix \ref{proof:1sttermraf} and Appendix \ref{proof:2ndtermraf}, respectively.

%

\begin{lemma}
	\label{le:1sttermraf}
	Fix any $\eta,\,\beta>0$. For any sufficiently small constant $\epsilon>0$, the following holds with probability at least $1-2{\rm e}^{-c_5\epsilon^2 m}$:
	\begin{equation}
	\label{eq:1sttermraf}
	\frac{1}{m}\sum_{i=1}^mw_i(\bm{a}_i^\ast\bm{h})^2\ge \frac{1-\zeta_1-\epsilon}{1+\beta(1+\eta)}\big\|\bm{h}\big\|^2
	\end{equation}
	with $w_i=\frac{1}{1+\nicefrac{\beta}{(\nicefrac{|\bm{a}_i^\ast\bm{z}|}{|\bm{a}_i^\ast\bm{x}|})}}$ for all $1\le i\le m$,  
	provided that $\nicefrac{m}{n}>(c_6 \cdot\epsilon^{-2}\log\epsilon^{-1})$ for certain numerical constants $c_5,\,c_6>0$. 
\end{lemma}



Now we turn to the second term in \eqref{eq:lrcfinal}. For ease of exposition, let us first introduce the following events
\begin{align}
\mathcal{B}_i&:=\big\{|\bm{a}_i^\ast\bm{x}|< |\bm{a}_i^
\ast\bm{h}|\le (k+1)|\bm{a}_i^\ast\bm{x}|
\big\}\label{eq:bset}\\
\mathcal{O}_i&:=\big\{ (k+1)|\bm{a}_i^\ast\bm{x}|< |\bm{a}_i^
\ast\bm{h}|
\big\}\label{eq:oset}
\end{align}
for all $1\le i\le m$ and some fixed constant $k>0$.
The second term can be bounded as follows
\begin{align}
\frac{1}{m}\sum_{i=1}^m2w_i\big|\bm{a}_i^\ast\bm{x}\big|\big|\bm{a}_i^\ast\bm{h}\big|\mathbb{1}_{\mathcal{D}_i}&\le \frac{1}{m}\sum_{i=1}^mw_i\!\left[(\bm{a}_i^\ast\bm{x})^2+(\bm{a}_i^\ast\bm{h})^2\right]\mathbb{1}_{\{(\bm{a}_i^\ast\bm{z})(\bm{a}_i^\ast\bm{x})<0\}}\nonumber\\
&=\frac{1}{m}\sum_{i=1}^m w_i\!\left[(\bm{a}_i^\ast\bm{x})^2+(\bm{a}_i^\ast\bm{h})^2\right]\mathbb{1}_{\{(\bm{a}_i^\ast\bm{h})(\bm{a}_i^\ast\bm{x})+(\bm{a}_i^\ast\bm{x})^2<0\}}\nonumber\\
&\le \frac{1}{m}\sum_{i=1}^m w_i\!\left[(\bm{a}_i^\ast\bm{x})^2+(\bm{a}_i^\ast\bm{h})^2\right]\mathbb{1}_{ \{|\bm{a}_i^\ast\bm{x}|< |\bm{a}_i^
	\ast\bm{h}|\}}\nonumber\\
&\le  \frac{2}{m}\sum_{i=1}^m w_i(\bm{a}_i^\ast\bm{h})^2\mathbb{1}_{ \{|\bm{a}_i^\ast\bm{x}|< |\bm{a}_i^
	\ast\bm{h}|\}}\nonumber\\
&= \frac{2}{m}\sum_{i=1}^m w_i(\bm{a}_i^\ast\bm{h})^2\mathbb{1}_{ \{|\bm{a}_i^\ast\bm{x}|< |\bm{a}_i^
	\ast\bm{h}|\le (k+1)|\bm{a}_i^\ast\bm{x}|\}}\nonumber\\
& + \frac{2}{m}\sum_{i=1}^m w_i(\bm{a}_i^\ast\bm{h})^2\mathbb{1}_{ \{(k+1)|\bm{a}_i^\ast\bm{x}|< |\bm{a}_i^
	\ast\bm{h}|\}}\nonumber\\
&=	\frac{2}{m}\sum_{i=1}^m w_i(\bm{a}_i^\ast\bm{h})^2\mathbb{1}_{ \mathcal{B}_i}
+	 \frac{2}{m}\sum_{i=1}^m w_i(\bm{a}_i^\ast\bm{h})^2\mathbb{1}_{ \mathcal{O}_i}
\label{eq:2ndbound}
\end{align}
where the first equality is derived by substituting $\bm{z}=\bm{h}+\bm{x}$ according to the definition of $\bm{h}$, the second event suffices for $(\bm{a}_i^\ast\bm{h})(\bm{a}_i^\ast\bm{x})+(\bm{a}_i^\ast\bm{x})^2<0$, and the second equality follows from writing the indicator function $\mathbb{1}_{\{|\bm{a}_i^\ast\bm{x}|< |\bm{a}_i^\ast\bm{h}|\}}$ as the summation of two indicator functions of two events $\mathbb{1}_{\{|\bm{a}_i^\ast\bm{x}|<|\bm{a}_i^\ast\bm{h}|\le (k+1)|\bm{a}_i^\ast\bm{x}|\}}$ and $\mathbb{1}_{\{|\bm{a}_i^\ast\bm{h}|> (k+1)|\bm{a}_i^\ast\bm{x}|\}}$.

The task so far remains to derive upper bounds for the two terms on the right side of \eqref{eq:2ndbound}, which leads to Lemma \ref{le:2ndtermraf}. 

\begin{lemma}
	\label{le:2ndtermraf}
	Fixing some $k>0$, define $\zeta_2$ to be the maximum of $\mathbb{E}[w_i]$ in \eqref{eq:comp} for $\varrho=0.01$ and $\nu =0.1$, which depends only on $k$.	
	For any $\epsilon>0$, if $\nicefrac{m}{n}>c_6 \epsilon^{-2}\log \epsilon^{-1}$, the following hold simultaneously with probability at least $1-c_3{\rm e}^{-c_2\epsilon^2m}$:
	\begin{equation}
	\frac{1}{m}\sum_{i=1}^m w_i(\bm{a}_i^\ast\bm{h})^2\mathbb{1}_{\mathcal{O}_i}\le (\zeta_2+\epsilon )\|\bm{h}\|^2\label{eq:2ndterm2raf}
	\end{equation}
	and
	\begin{equation}
	\frac{1}{m}\sum_{i=1}^m w_i(\bm{a}_i^\ast\bm{h})^2\mathbb{1}_{ \mathcal{B}_i}\le \frac{0.1271-\zeta_2+\epsilon}{1+\nicefrac{\beta}{k}}\|\bm{h}\|^2 \label{eq:2ndterm1raf}
	\end{equation}
	for all $\bm{h}\in\mathbb{R}^n$ obeying $\nicefrac{\|\bm{h}\|}{\|\bm{x}\|}\le \nicefrac{1}{10}$, where $c_1,\,c_2,\,c_3>0$ are some universal constants. 
\end{lemma}


Taking the results in \eqref{eq:1sttermraf}, \eqref{eq:2ndbound}, and \eqref{eq:2ndterm2raf}-\eqref{eq:2ndterm1raf}  established in Lemmas \ref{le:1sttermraf} and \ref{le:2ndtermraf} back into \eqref{eq:lrcfinal}, we conclude that
\begin{align}
\langle\ell_{\rm rw}(\bm{z}),\bm{h}\rangle& \ge \frac{1}{m}\sum_{i=1}^mw_i(\bm{a}_i^\ast\bm{h})^2\mathbb{1}_{\mathcal{E}_i}
-\frac{1}{m}\sum_{i=1}^m2w_i\big|\bm{a}_i^\ast\bm{x}\big|\big|\bm{a}_i^\ast\bm{h}\big|\mathbb{1}_{\mathcal{D}_i}\nonumber\\
&\ge \frac{1-\zeta_1-\epsilon}{1+\beta(1+\eta)}\|\bm{h}\|^2-2(\zeta_2+\epsilon )\|\bm{h}\|^2-\frac{2(0.1271-\zeta_2+\epsilon)}{1+\nicefrac{\beta}{k}}\|\bm{h}\|^2\nonumber\\
&=\left[\frac{1-\zeta_1-\epsilon}{1+\beta(1+\eta)}-2(\zeta_2+\epsilon )-\frac{2(0.1271-\zeta_2+\epsilon)}{1+\nicefrac{\beta}{k}}
\right]\|\bm{h}\|^2\label{eq:final}
\end{align} 
which will be rendered positive, provided that $\beta>0$ is small enough, and that parameters
$\eta,\,k>0$ are suitably chosen.


\subsection{Proof of Lemma~\ref{le:1sttermraf}}\label{proof:1sttermraf}

Plugging in the weighting parameters $w_i=\frac{1}{1+\nicefrac{\beta}{(\nicefrac{|\bm{a}_i^\ast\bm{z}|}{|\bm{a}_i^\ast\bm{x}|})}}$ and based on the definition of $\mathcal{E}_i$, the first term in \eqref{eq:lrcfinal} can be lower bounded as follows
\begin{align}
\frac{1}{m}\sum_{i=1}^mw_i(\bm{a}_i^\ast\bm{h})^2&\ge \frac{1}{m}\sum_{i=1}^m\frac{1}{1+\nicefrac{\beta}{(\nicefrac{|\bm{a}_i^\ast\bm{z}|}{|\bm{a}_i^\ast\bm{x}|})}}(\bm{a}_i^\ast\bm{h})^2\mathbb{1}_{\mathcal{E}_i}\label{eq:1stineq}\\
&\ge \frac{1}{m}\sum_{i=1}^m\frac{1}{1+\beta(1+\eta)}(
\bm{a}_i^\ast\bm{h})^2\mathbb{1}_{\big\{\frac{|\bm{a}_i^\ast\bm{z}|}{|\bm{a}_i^\ast\bm{x}|}\ge \frac{1}{1+\eta}\big\}}\nonumber\\
&=\frac{1}{1+\beta(1+\eta)}\cdot\frac{1}{m}\sum_{i=1}^m(
\bm{a}_i^\ast\bm{h})^2\mathbb{1}_{\mathcal{E}_i}
\label{eq:1stineq}
\end{align}
where the first inequality arises from dropping some nonnegative terms from the left hand side, and the second one replaced the ratio $\nicefrac{|\bm{a}_i^\ast\bm{z}|}{|\bm{a}_i^\ast\bm{x}|}$ in the weights by its lower bound $\nicefrac{1}{1+\eta}$ because the weights are monotonically increasing functions of the ratios $\nicefrac{|\bm{a}_i^\ast\bm{z}|}{|\bm{a}_i^\ast\bm{x}|}$.
Using the result in Lemma \ref{le:1sttermtaf}, the last term in \eqref{eq:1stineq} can be further bounded by 
\begin{equation}
\frac{1}{m}\sum_{i=1}^mw_i(\bm{a}_i^\ast\bm{h})^2\ge \frac{1}{1+\beta(1+\eta)}\cdot\frac{1}{m}\sum_{i=1}^m(
\bm{a}_i^\ast\bm{h})^2\mathbb{1}_{\mathcal{E}_i}\ge \frac{1-\zeta_1-\epsilon}{1+\beta(1+\eta)}\|\bm{h}\|^2
\end{equation}
for any fixed sufficiently small constant $\epsilon>0$, which holds with probability at least $1-2{\rm e}^{-c_5\epsilon^2 m}$, if $m>(c_6 \cdot\epsilon^{-2}\log\epsilon^{-1})n$.

\subsection{Proof of Lemma~\ref{le:2ndtermraf}}\label{proof:2ndtermraf}
The proof is adapted from that of \cite[Lemma 9]{reshaped}. We first prove the bound \eqref{eq:2ndterm2raf} for any fixed $\bm{h}$ obeying $\|\bm{h}\|\le \nicefrac{\|\bm{x}\|}{10}$, and subsequently develop a uniform bound at the end of this section. The bound \eqref{eq:2ndterm1raf} can be derived directly from subtracting the bound in \eqref{eq:2ndterm2raf} with $k$ from that bound with $k=0$, followed by an application of the Bernstein-type sub-exponential tail bound \cite{chap2010vershynin}. Hence, we only discuss the first bound \eqref{eq:2ndterm2raf}. 
Because of the discontinuity hence non-Lipschitz of the indicator functions, let us approximate them by a sequence of auxiliary Lipschitz functions. Specifically, with some constant $\varrho>0$, define for all $1\le i\le m$ the ensuing continuous functions
\begin{equation}
\label{eq:chiraf}
\chi_i(s):=\left\{ \begin{array}{lll}
s,&&s>(1+k)^2(\bm{a}_i^\ast\bm{x})^2\\
\frac{1}{\varrho}[s-(k+1)^2(\bm{a}_i^\ast\bm{x})^2]\atop+(k+1)^2(\bm{a}_i^\ast\bm{x})^2,&& (1-\varrho)(k+1)^2(\bm{a}_i^\ast\bm{x})^2\le s \le (k+1)^2(\bm{a}_i^\ast\bm{x})^2\\
0,&&{\rm otherwise}.
\end{array}
\right.
\end{equation}
Clearly, all $\chi_i$(s)'s are random Lipschitz functions with constant $\nicefrac{1}{\varrho}$. Furthermore, it is easy to verify that
\begin{equation}
|\bm{a}_i^\ast\bm{h}|^2\mathbb{1}_{\{(k+1)|\bm{a}_i^\ast\bm{x}|<|\bm{a}_i^\ast\bm{h}| \}}\le \chi_i(|\bm{a}_i^\ast\bm{h}|^2)\le |\bm{a}_i^\ast\bm{h}|^2\mathbb{1}_{\{\sqrt{1-\varrho}(k+1)|\bm{a}_i^\ast\bm{x}|<|\bm{a}_i^\ast\bm{h}|
	\}}.
\end{equation}

Given that the second term involves the addition event $\mathcal{G}_i$ in \eqref{eq:gset}, define $w_i:=\frac{|\bm{a}_i^\ast\bm{h}|^2}{\|\bm{h}\|^2}\mathbb{1}_{\{\sqrt{1-\varrho}(k+1)|\bm{a}_i^\ast\bm{x}|<|\bm{a}_i^\ast\bm{h}|
	\}}$ for $1\le i\le m$, and also $\nu :=\frac{\|\bm{h}\|}{\|\bm{x}\|}$ for notational convenience. If $f(\tau_1,\tau_2)$ denotes the density of two joint Gaussian random variables with correlation constant $\rho =\frac{\bm{h}^\ast\bm{x}}{\|\bm{h}\|\|\bm{x}\|}\in (-1,1)$, then
the expectation of $w_i$ can be obtained based on the conditional expectation
\begin{align}
\mathbb{E}[w_i]&=\int_{-\infty}^{\infty}\mathbb{E}[w_i|\bm{a}_i^\ast\bm{x}=\tau_1\|\bm{x}\|,\bm{a}_i^\ast\bm{h}=\tau_1\|\bm{h}\|] f(\tau_1,\tau_2) {\rm d}\tau_1 {\rm d}\tau_2\nonumber\\
&=\int_{-\infty}^\infty\int_{-\infty}^\infty \tau_2^2 \mathbb{1}_{ \{\sqrt{1-\varrho}(k+1)|\tau_1|<|\tau_2|\nu
	\}
} f(\tau_1,\tau_2) {\rm d}\tau_1 {\rm d}\tau_2\nonumber\\
&=\frac{1}{\sqrt{2\pi}}\int_0^\infty \tau^2_2 \exp(\nicefrac{-\tau_2^2}{2})\bigg[{\rm erf}\left(
\frac{(\nicefrac{\nu}{[\sqrt{1-\varrho}(k+1)]}-\rho)\tau_2
}{\sqrt{2(1-\rho^2)}}	\right)\nonumber\\
&\quad +{\rm erf}\left(
\frac{(\nicefrac{\nu}{[\sqrt{1-\varrho}(k+1)]}+\rho)\tau_2
}{\sqrt{2(1-\rho^2)}}	\right) 
\bigg] {\rm d}\tau_2
\label{eq:comp}\\
&:=\zeta_2\label{eq:zeta2}.
\end{align}

It is not difficult to see that $\mathbb{E}[w_i]=0$ for $\rho =\pm 1$, and 
$\mathbb{E}[w_i]$ is continuous over $\rho\in(-1,1)$ due to the integration property of continuous functions over a continuous interval. Although the last term in \eqref{eq:comp} can not be expressed in closed-form, it can be evaluated numerically. Note first that for fixed parameters $\varrho>0$ and $\nu\le 0.1$, the integration above is monotonically decreasing in $k\ge 0$, and achieves the maximum at $k=0$.  
For parameter values $k=5$, $\nu=0.1$ and $\varrho=0.01$, Fig. \ref{fig:wvalue} plots $\mathbb{E}[w_i]$ as a function of $\rho$, whose maximum $
\zeta_2=0.0213.
$ is achieved at $\rho=0$.
Further from the integration in \eqref{eq:comp}, for fixed $k\ge 0$, $\mathbb{E}[w_i]$ is a monotonically increasing function of both $\nu$ and $\varrho$, it is therefore safe to conclude that
for all $0<\nu\le 0.1$, and $\varrho=0.01$, we have
\begin{equation}
\mathbb{E}[w_i]\le \zeta_2=0.0213.
\end{equation}
Hence, it is safe to conclude that $\mathbb{E}[\chi_i(|\bm{a}_i^\ast\bm{h}|^2)]\le 0.0213\|\bm{h}\|^2$ for $\nu<0.1$, $\varrho=0.01$, and $k=5$. Since $[\chi_i(|\bm{a}_i^\ast\bm{h}|^2$'s are sub-exponential with sub-exponential norm of 
the order $\mathcal{O}(\|\bm{h}\|^2)$, Bernstein-type sub-exponential tail bound \cite{chap2010vershynin} confirms that
\begin{equation}
\mathbb{p}\!\left(\frac{1}{m}\sum_{i=1}^m\frac{\chi_i(|\bm{a}_i^\ast\bm{h}|^2)}{\|\bm{h}\|^2}
>(\zeta_2+\epsilon)\right)<{\rm e}^{-c_7m\epsilon^2}
\end{equation}
for some numerical constant $\epsilon>0$, provided that $\|\bm{h}\|\le \nicefrac{\|\bm{x}\|}{10}$. Finally, due to the fact that $w_i\le 1$ for all $1\le i\le m$, the following holds
\begin{equation}
\frac{1}{m}\sum_{i=1}^mw_i{\chi_i(|\bm{a}_i^\ast\bm{h}|^2)}
<(\zeta_2+\epsilon)\|\bm{h}\|^2
\end{equation}
with probability at least $1-{\rm e}^{-c_7m\epsilon^2}$.

\begin{figure}[ht]
	\begin{center}
		\centerline{\includegraphics[scale = 0.5]{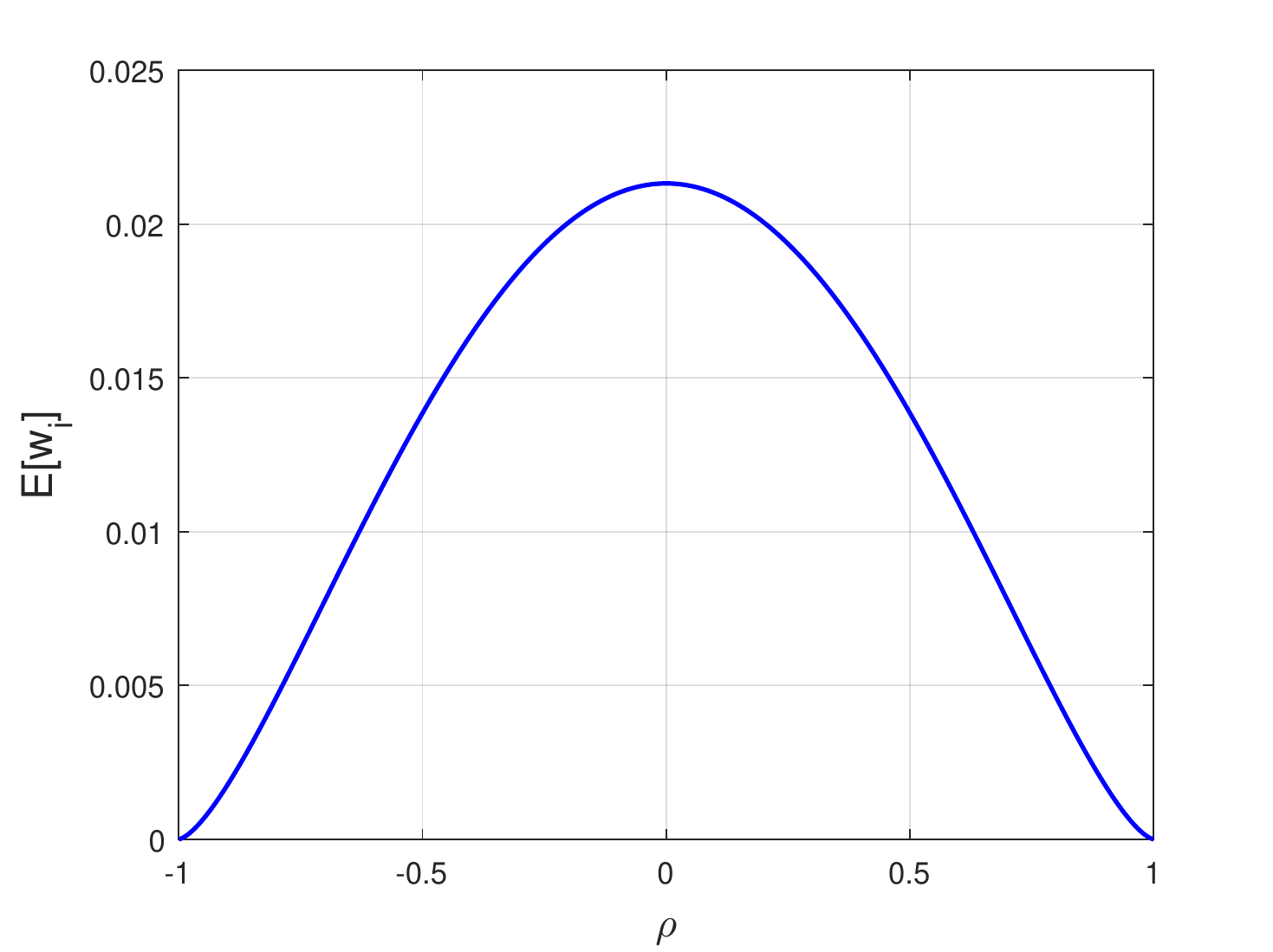}}
		\vspace{-0em}
		\caption{The expectation $\mathbb{E}[w_i]$ as a function of $\rho$ over $[-1,1]$. 
		}
			\vspace{-.0cm}
		\label{fig:wvalue}
	\end{center}
\end{figure}

We have proved the bound in \eqref{eq:2ndterm2raf} for a fixed vector $\bm{h}$, and the uniform bound for all vectors $\bm{h}$ obeying $\|\bm{h}\|\le \nicefrac{\|\bm{x}\|}{10}$ can be obtained by similar arguments in the proof \cite[Lemma 9]{reshaped} with only minor changes in the constants.

Regarding the second bound \eqref{eq:2ndterm1raf}, it is easy to see that 
\begin{align}
\frac{1}{m}\sum_{i=1}^m|\bm{a}_i^\ast\bm{h}|^2\mathbb{1}_{\{|\bm{a}_i^\ast\bm{x}| <|\bm{a}_i^\ast\bm{h}| \le (k+1)|\bm{a}_i^\ast\bm{x}|\}}&=\frac{1}{m}\sum_{i=1}^m|\bm{a}_i^\ast\bm{h}|^2\mathbb{1}_{\{|\bm{a}_i^\ast\bm{x}|<|\bm{a}_i^\ast\bm{h}|\}}\nonumber\\
&-\frac{1}{m}\sum_{i=1}^m|\bm{a}_i^\ast\bm{h}|^2\mathbb{1}_{\{(k+1)|\bm{a}_i^\ast\bm{x}|<|\bm{a}_i^\ast\bm{h}|\}}\nonumber\\
&\le (0.1271-\zeta_2+\epsilon)\|\bm{h}\|^2\label{eq:2final2}
\end{align}
where the last inequality follows from subtracting the bound in \eqref{eq:2ndterm2raf} of $k$ from that corresponding to $k=0$. To account for the weights $w_i=\frac{1}{1+\nicefrac{\beta}{(\nicefrac{|\bm{a}_i^\ast\bm{z}|}{|\bm{a}_i^\ast\bm{x}|})}}$, first notice that $\bm{a}_i^\ast\bm{h}=\bm{a}_i^\ast\bm{z}-\bm{a}_i^\ast\bm{x}$, 
and that our second bound works with $(\bm{a}_i^\ast\bm{z})(\bm{a}_i^\ast\bm{x})<0$ in \eqref{eq:lrcfinal}, 
hence
$
\frac{|\bm{a}_i^\ast\bm{z}|}{|\bm{a}_i^\ast\bm{x}|} \le  \frac{|\bm{a}_i^\ast\bm{h}|}{|\bm{a}_i^\ast\bm{x}|}-1
$. Recall that the second bound \eqref{eq:2ndterm1raf} assumes the event 
$\{|\bm{a}_i^\ast\bm{x}| <|\bm{a}_i^\ast\bm{h}| \le (k+1)|\bm{a}_i^\ast\bm{x}|\}$, implying 
$\frac{|\bm{a}_i^\ast\bm{z}|}{|\bm{a}_i^\ast\bm{x}|} \le  \frac{|\bm{a}_i^\ast\bm{h}|}{|\bm{a}_i^\ast\bm{x}|}-1\le k$. 
Further, because $w_i$ is monotonically increasing in $\frac{|\bm{a}_i^\ast\bm{z}|}{|\bm{a}_i^\ast\bm{x}|} $, then $w_i\le \frac{1}{1+\nicefrac{\beta}{k}}$. Taking this result back to \eqref{eq:2final2} yields
\begin{equation}
\frac{1}{m}\sum_{i=1}^mw_i|\bm{a}_i^\ast\bm{h}|^2\mathbb{1}_{\{|\bm{a}_i^\ast\bm{x}| <|\bm{a}_i^\ast\bm{h}| \le (k+1)|\bm{a}_i^\ast\bm{x}|\}}
\le \frac{0.1271-\zeta_2+\epsilon}{1+\nicefrac{\beta}{k}}\|\bm{h}\|^2
\end{equation}
which proves the second bound in \eqref{eq:2ndterm1raf}.

%
%

\begin{lemma} (\cite[Lemma 5]{taf})
	\label{le:1sttermtaf}
	Fix $\eta\ge \nicefrac{1}{2}$ and $\rho\le \nicefrac{1}{10}$, and let $\mathcal{E}_i$ be defined in~\eqref{eq:eset}.\\ For independent random variables $Y\sim\mathcal{N}(0,\,1)$ and $Z\sim\mathcal{N}(0,\,1)$, define
	\begin{align}
	\label{eq:zeta12}
	\zeta_1:=1-\min\left\{
	\mathbb{E}\left[\mathbb{1}_{\left\{\left|\frac{1-\rho}{\rho}+\frac{Y}{Z}
		\right|\ge \frac{\sqrt{1.01}}{\rho\left(1+\eta\right)}
		\right\}
	}
	\right],\;\mathbb{E}\left[Z^2\mathbb{1}_{\left
		\{\left|\frac{1-\rho}{\rho}+\frac{Y}{Z}
		\right|\ge \frac{\sqrt{1.01}}{\rho\left(1+\eta\right)}\right\}
	}
	\right]
	\right\}.
	\end{align}
	Fixing any $\epsilon>0$ and for any $\bm{h}$ satisfying $\nicefrac{\|\bm{h}\|}{\|\bm{x}\|}\le \rho$, the next holds
	with probability $1-2{\rm e}^{-c_5\epsilon^2 m}$:
	\begin{equation}\label{eq:1stterm}
	\frac{1}{m}\sum_{i=1}^m\left(\bm{a}_i^\ast\bm{h}\right)^2\mathbb{1}_{\mathcal{E}_i}\ge \left(
	1-\zeta_1-\epsilon\right)\left\|\bm{h}\right\|^2
	\end{equation}
	provided that $m>(c_6 \cdot\epsilon^{-2}\log\epsilon^{-1})n$ for some universal constants $c_5,\,c_6>0$.  
\end{lemma}
To have an estimate of the size of $\xi_1$ in~\eqref{eq:zeta12}, 
if $\gamma=0.7$ and $\rho=1/10$, we have
$\mathbb{E}\left[\mathbb{1}_{\left\{\left|\frac{1-\rho}{\rho}+\frac{Y}{Z}
	\right|\ge \frac{\sqrt{1.01}}{\rho\left(1+\gamma\right)}
	\right\}
}
\right]\approx 0.9216$, and $\mathbb{E}\left[Z^2\mathbb{1}_{\left
	\{\left|\frac{1-\rho}{\rho}+\frac{Y}{Z}
	\right|\ge \frac{\sqrt{1.01}}{\rho\left(1+\gamma\right)}\right\}
}
\right]\approx 0.9908$, hence leading to $\zeta_1\approx 0.0784$.


\vspace{1cm}

\bibliographystyle{IEEEtranS}


\bibliography{apower}

\end{document}